\documentclass{amsart}
  \usepackage{amsthm} %%% Defines \theoremstyle
  \usepackage{amsmath} %%% Defines \numberwithin, \operatorname
  \usepackage{amssymb} %%% Defines \mathbb
  \usepackage{mathtools} %%% Defines \eqqcoon and lots of symbols and arrows
   \usepackage{amsfonts}
  \usepackage{color}
\usepackage{stackengine}
   
  \usepackage[cal=boondox]{mathalfa}

 \usepackage{graphicx}
  \usepackage{xspace}
  \usepackage[all]{xy}
  \usepackage{pinlabel}
  \usepackage{enumerate}
\usepackage{appendix}

\usepackage[margin=1.3in]{geometry}
  \usepackage{hyperref}

  \newcommand{\RR}{\mathbb{R}}

  \newtheorem{theorem}{Theorem}[section]

  \newtheorem{lemma}[theorem]{Lemma}

  \theoremstyle{definition}
  \newtheorem{definition}[theorem]{Definition}

  \newtheorem*{claim*}{Claim}
  
  \newtheorem{example}[theorem]{Example}
  
  \newtheorem*{question*}{Question}
  \newtheorem*{answer*}{Answer}
  \newtheorem*{application*}{Application}

  \theoremstyle{remark}
  \newtheorem{remark}[theorem]{Remark}
  \newtheorem*{remark*}{Remark}
  
  \newcommand{\PSL}{\ensuremath{\operatorname{PSL}}\xspace} 
  \newcommand{\Map}{\ensuremath{\operatorname{Map}}\xspace}    
  \newcommand{\param}{{\mathchoice{\mkern1mu\mbox{\raise2.2pt\hbox{$
  \centerdot$}}
  \mkern1mu}{\mkern1mu\mbox{\raise2.2pt\hbox{$\centerdot$}}\mkern1mu}{
  \mkern1.5mu\centerdot\mkern1.5mu}{\mkern1.5mu\centerdot\mkern1.5mu}}}
  \newcommand{\from}{\colon\thinspace}

  \DeclarePairedDelimiter\abs{\lvert}{\rvert}

\begin{document}

\title{A construction of pseudo-Anosov homeomorphisms using positive twists}

%authors
 \author   {Yvon Verberne}
 \address{Department of Mathematics, University of Toronto, Toronto, ON }
 \email{yvon.verberne@mail.utoronto.ca}
 
%\ thanks{The first Authors thanks whoever helped.}
 
  \date{\today}
%  \subjclass[2010]{57M50 (32G15, 37D40, 37A25)}

\begin{abstract}
We introduce a construction of pseudo-Anosov homeomorphisms on $n$-times punctured spheres and surfaces with higher genus using only sufficiently many positive half-twists. These constructions can produce explicit examples of pseudo-Anosov maps with various number-theoretic properties associated to the stretch factors, including examples where the trace field is not totally real and the Galois conjugates of the stretch factor are on the unit circle.
\end{abstract}

\maketitle

%---------------------------------------------------------------------------------------------------------------------%
\section{Introduction}

Let $S_{g,n}$ be a surface of genus $g$ with $n$ punctures. The mapping class group, $\Map(S_{g,n})$, is the group of isotopy classes of orientation-preserving homeomorphisms of $S_{g,n}$. The Nielsen-Thurston classification of elements states that each element in $\Map(S_{g,n})$ is either periodic, reducible, or pseudo-Anosov.

Thurston proved the Nielsen-Thurston classification and provided us with the definition of pseudo-Anosov mapping classes \cite{Th}. Thurston defined an element $f \in \Map(S_{g,n})$ to be pseudo-Anosov if there is a representative homeomorphism $\phi$, a number $\lambda > 1$ and a pair of transverse measured foliations $\mathcal{F}^{\mathcal{u}}$ and $\mathcal{F}^{\mathcal{s}}$ such that $\phi(\mathcal{F}^{\mathcal{u}}) = \lambda \mathcal{F}^{\mathcal{u}}$ and $\phi(\mathcal{F}^{\mathcal{s}}) = \lambda^{-1} \mathcal{F}^{\mathcal{s}}$. $\lambda$ is called the \textit{stretch factor} (or dilatation) of $f$, $\mathcal{F}^{\mathcal{u}}$ and $\mathcal{F}^{\mathcal{s}}$ are the \textit{unstable foliation} and \textit{stable foliation}, respectively, and the map $\phi$ is a \textit{pseudo-Anosov homeomorphism}.

\begin{figure}[h]
\begin{picture}(100,160)
\put(-50,-25){\includegraphics{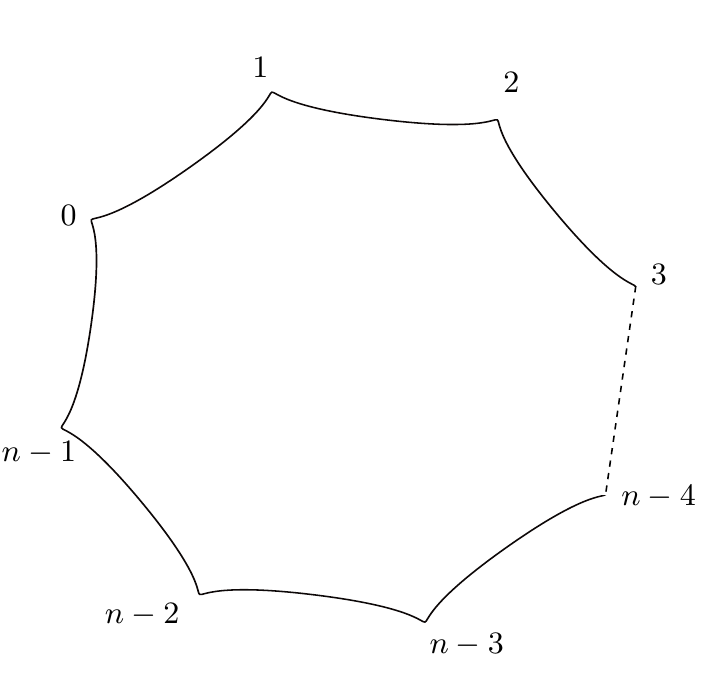}}
\end{picture}
\caption{Labelling of the punctures on the n-times punctured sphere.}
\label{Fig:labellingsphere} 
\end{figure}

In this paper, we will provide a new construction of pseudo-Anosov mapping classes on n-times punctured spheres. In particular, consider the n-times punctured sphere, $S_{0,n}$, with a clockwise labelling of the punctures, as depicted in Figure \ref{Fig:labellingsphere}. 

In this paper, we will consider the following curves on the $n$-times punctured sphere: Consider the plane with $n$ punctures, where the punctures are located at the vertices of a regular $n$-gon. Then for any subset of the punctures, there is a unique isotopy class of curves that surrounds those punctures and is convex in the Euclidean metric. This curve gives a curve on the sphere with $n$ punctures via the stereographic projection. This is well defined up to change of coordinates. Using these curves, it is possible to associate a half-twist to each puncture on $S_{0,n}$.

\begin{definition}
Consider a simple closed curve described above. We say that such a curve $\gamma$ separates punctures $k$ and $l$ from $S_{0,n}$ if one of the subsurfaces obtained by cutting along $\gamma$ contains only the punctures $k$ and $l$, and the other subsurface contains the remaining punctures. Denote the curve separating puncture $j$ and $j-1 \mod n$ by $\alpha_j$. Define the half-twist associated to puncture $j$, denoted $D_{j}$, as the half-twist around $\alpha_j$. Two subsequent half-twists, $D_{j}^2$, is called a Dehn twist.
\end{definition}

In this paper, our choice of a positive half-twist will be a right half-twist.

Consider a set of punctures $p = \{ p_1, \ldots , p_i \}$ such that $|p_k - p_j| \geq 2 \mod n$ for each $j , k \in \{ 1, \ldots , i \}$. Then every curve associated to a puncture from this set is disjoint from any other curve associated to a puncture from this set. As the curves in $p$ are disjoint, we are able to perform the half-twists associated to the curves in $p$ simultaneously as a multi-twist. We denote this multi-twist by $D_p = D_{p_i} \ldots D_{p_1}$. 

Additionally, we define the following map:

\begin{definition}
Define the map $\rho$ as follows:
\begin{equation*}
\begin{aligned}
\rho \from \mathbb{Z}_{n} &\to \mathbb{Z}_{n}\\
j &\mapsto j+1 \mod n.
\end{aligned}
\end{equation*}
This is the map which permutes the numbers $1, \ldots, n$ cyclically.
\end{definition}

Using the above notation, we are able to define pseudo-Anosov maps on $S_{0,n}$ by considering the different partitions of the punctures of the sphere. We recall that a \textit{partition} of a set is a grouping of the set's elements into non-empty subsets in such a way that every element is included in one and only one of the subsets. We define a partition of the punctures, $\mu = \{ \mu_1, \ldots, \mu_k \}$ to be \textit{evenly spaced} if $\rho(\mu_i)=\mu_{i+1}$ for $1 \leq i \leq k$. Notice that if a partition is evenly spaced that $|\mu_i|=|\mu_j|$ for all $\mu_i, \mu_j \in \mu$.

\begin{example}\label{partitioningthepunctures}
Consider the 6-times punctured sphere.
\begin{figure}[h]
\setlength{\unitlength}{0.01\linewidth}
\begin{picture}(45,38)
\put(3,0){\includegraphics{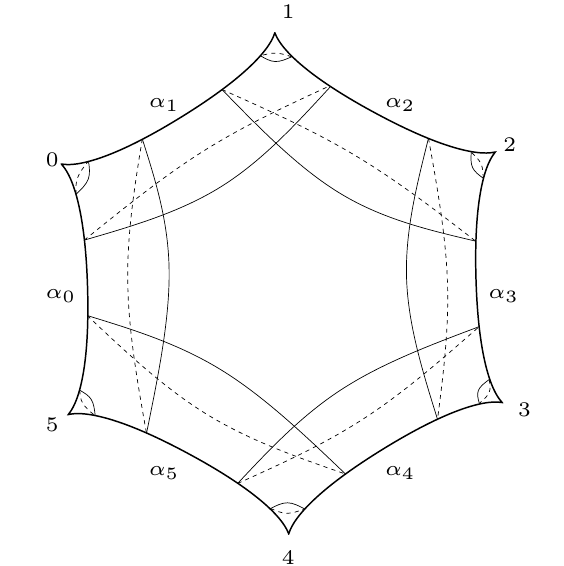}}
\end{picture}
\caption{The 6-times punctured sphere.}
\label{Fig:sixtimespunctured} 
\end{figure}
We label the punctures in a clockwise fashion. Up to spherical symmetry, we notice that we may partition the punctures evenly in two different ways. First, we could place every third puncture in the same set to get $\mu = \{ \{0,3\}, \{1,4\}, \{3,5\}\}=\{\mu_1,\mu_2,\mu_3\}$, and secondly we could place every other puncture in the same set to get $\bar{\mu} = \{ \{0,2,4\}, \{1,3,5\} \} = \{ \bar{\mu}_1, \bar{\mu}_2 \}$.
\end{example}

\begin{theorem}\label{partitions}
Consider the surface $S_{0,n}$. Let $\{ \mu_i \}_{i=1}^{k}$, for $1<k<n$, be an evenly spaced partition of the punctures of $S_{0,n}$. Then 
\[
\phi = \prod_{i=1}^{k} D_{\mu_i}^{q_i} = D_{\mu_k}^{q_k} \ldots D_{\mu_2}^{q_2} D_{\mu_1}^{q_1},
\]
where $q_j = \{q_{j_1}, \ldots q_{j_l} \}$ are tuples of integers greater than one, is a pseudo-Anosov mapping class.
\end{theorem}

\begin{example}
Returning to the partitions from Example \ref{partitioningthepunctures}, we notice that Theorem \ref{partitions} tells us that the maps $\phi = D_{5}^2D_{2}^2D_{4}^2D_{1}^2D_{3}^2D_{0}^2=D_{\mu_3}^2D_{\mu_2}^2D_{\mu_1}^2$, and $\bar{\phi} = D_{5}^2D_{3}^2D_{1}^2D_{4}^2D_{2}^2D_{0}^2=D_{\bar{\mu}_2}^2D_{\bar{\mu}_1}^2$ are pseudo-Anosov mapping classes. We will prove this in detail in Section \ref{Sec:Examples}.
\end{example}

One can perform a modification of the partitions of Theorem \ref{partitions} in order to construct additional pseudo-Anosov mapping classes.

\begin{theorem}\label{modification}
Consider the surface $S_{0,n}$. Consider one of the maps from Theorem \ref{partitions}, in particular, consider a partition of the $n$ punctures into $1<k<n$ sets $\{ \mu_i \}_{i=1}^{k}$ such that the partition is evenly spaced. We modify this partition to construct a map for the $(n+1)$-times punctured sphere as follows: Create a new partition, $\mu '$, by replacing each label $j \geq k+1$ by $j+1$ to obtain the sets $\{ \mu_i ' \}_{i=1}^{j}$ and setting the $(k+1)^{\text{st}}$ partition to be $\mu_{k+1} ' = \{ k+1\}$. Then 
\[
\phi '= \prod_{i=1}^{k} D_{\mu_i'}^{q_i '} = D_{\mu_{k+1} '}^{q_{k+1} '} D_{\mu_k '}^{q_{k} '} \ldots D_{\mu_2 '}^{q_2 '} D_{\mu_1 '}^{q_{1} '},
\]
where $q_j ' = \{q_{j_1} ', \ldots q_{j_l} ' \}$ are tuples of integers greater than one, is a pseudo-Anosov mapping class.
\end{theorem}

\begin{theorem}\label{uniqueconstruction}
There exists a pseudo-Anosov mapping class obtained from the construction of Theorem \ref{modification} which cannot be obtained from either the Thurston or the Penner constructions.
\end{theorem}

\begin{remark}[The Braid Group]
There is a well defined map from the braid group on $n$ strands, $B_n$, to the mapping class group of a disk with $n$ punctures, $\Map(D_n)$ (see Section 9.1.3 of \cite{FM} for a detailed description). In fact, the map $B_n \to \Map(D_n)$ is an isomorphism. By the Alexander trick, any homeomorphism of the disk which fixes the boundary is isotopic to the identity. Therefore, there is a canonical embedding $B_n \hookrightarrow \Map(S_{n+1})$ , where one of the punctures of $\Map(S_{0,n+1})$ is a marked point \cite{LT}. Additionally, there is the forgetful map from $\Map(S_{0,n+1}) \to \Map(S_{0,n})$, where we ``fill in" the marked point in $S_{0,n+1}$. This is the same forgetful map as in Birman's exact sequence. We are able to compose these two maps to obtain a map from the braid group to the $n$-times punctured sphere. It follows that if a map in $\Map(S_{n})$ is pseudo-Anosov, the corresponding map in $B_n$ is pseudo-Anosov.
\end{remark}

One elementary construction of pseudo-Anosov homeomorphisms was provided by Thurston \cite{Th}. We recall that parabolic isometries correspond to those non-identity elements of $\PSL(2,\mathbb{R})$ with trace $\pm 2$. We recall that $A$ is a multi-curve if A is the union of a finite collection of disjoint simple closed curves in $S$.

\begin{theorem}[Thurston]
Suppose that $A = \{\alpha_1, \ldots, \alpha_n \}$ and $B = \{ \beta_1, \ldots, \beta_m \}$ are multicurves in $S$ so that $A$ and $B$ are filling, that is, $A$ and $B$ are in minimal position and the complement of $A \cup B$ is a union of disks and once punctured disks. There is a real number $\mu = \mu(A,B)$, and a representation $\rho \from \left< D_A, D_B \right> \to \PSL(2, \mathbb{R})$ given by
\begin{equation*}
D_{\alpha} \mapsto \left(
\begin{matrix}
1 & \mu^{1/2}\\
0 & 1 
\end{matrix}
\right)
\end{equation*}
and
\begin{equation*}
D_{\beta} \mapsto \left(
\begin{matrix}
1 & 0\\
-\mu^{1/2} & 1
\end{matrix}
\right).
\end{equation*}
with the following properties:
\begin{enumerate}
\item An element $f \in \left< D_A, D_B \right>$ is periodic, reducible, or pseudo-Anosov according to whether $\rho(f)$ is elliptic, parabolic, or hyperbolic.
\item When $\rho(f)$ is parabolic $f$ is a multitwist.
\item When $\rho(f)$ is hyperbolic the pseudo-Anosov homeomorphism $f$ has stretch factor equal to the larger of the two eigenvalues of $\rho(f)$.
\end{enumerate}
\end{theorem}

After the work of Thurston, Penner gave the following very general construction of pseudo-Anosov homeomorphisms \cite{Pe}:

\begin{theorem}[Penner]
Let $A=\{ a_1, \ldots, a_n \}$ and $B = \{ b_1, \ldots, b_m \}$ be a pair of multicurves on a surface $S$, and suppose that $A$ and $B$ are filling. Then any product of positive Dehn twists about $a_j$ and negative Dehn twists about $b_k$ is pseudo-Anosov provided that all $n+m$ Dehn twists appear in the product at least once.
\end{theorem}

There has been a considerable amount of research regarding the number-theoretic properties of the stretch factors of pseudo-Anosov maps. Hubert and Lanneau proved that for Thurston's construction, the field $\mathbb{Q} (\lambda + 1/\lambda)$ is always totally real \cite{HL}, and Shin and Strenner proved that for Penner's construction the Galois conjugates of the stretch factor are never on the unit circle \cite{SS}.

For the new pseudo-Anosov mapping classes constructed in this paper, it is shown that there is a large variety of number-theoretic properties associated to the stretch factors. In contrast to the above results for the maps from Penner and Thurston's constructions, it is possible to find explicit examples of pseudo-Anosov mapping classes using this new construction so that $\mathbb{Q} (\lambda + 1/\lambda)$ is not totally real and the Galois conjugates of the stretch factors are on the unit circle. However, these results are not able to be generalized as it is also possible to find explicit examples where $\mathbb{Q} (\lambda + 1/\lambda)$ is totally real and the Galois conjugates of the stretch factors are not on the unit circle, where $\mathbb{Q} (\lambda + 1/\lambda)$ is totally real and the Galois conjugates of the stretch factors are on the unit circle, and where $\mathbb{Q} (\lambda + 1/\lambda)$ is not totally real and the Galois conjugates of the stretch factors are not on the unit circle.

\subsection*{Outline} In Section \ref{Sec:Examples} we begin by presenting an explicit example of how to apply Theorem \ref{partitions} to construct two pseudo-Anosov mapping classes on $S_{0,6}$. We then apply Theorem \ref{modification} to these two constructions to construct two pseudo-Anosov mapping classes on $S_{0,7}$. In Section \ref{Sec:Main}, we will provide proofs that these constructions produce pseudo-Anosov mapping classes, as well as a generalization of the nesting lemma from \cite{MM1}. In Section \ref{Sec:Modifications} we will discuss some additional modifications which can be made to construct additional pseudo-Anosov mapping classes. Section \ref{Sec:HigherGenus} will discuss how the construction produces pseudo-Anosov mapping classes on surfaces of higher genus through a branched cover. Lastly, in \ref{Sec:Algebraic} we will provide explicit examples of pseudo-Anosov mapping classes to illustrate the variety of number-theoretic properties associated to the stretch factors, as well as proving that this construction is unique. 

\subsection*{Acknowledgements} I would like to thank Dan Margalit for suggesting I generalize the pseudo-Anosov map from \cite{RV}, Bal\'azs Strenner for suggesting I analyze the number-theoretic properties associated to the stretch factors of the maps produced, and Joan Birman for suggesting I apply the construction to the braid group and to surfaces of higher genus. I would also like to thank Thad Janisse, Chris Leininger, Dan Margalit, Kasra Rafi, Joe Repka, and Bal\'azs Strenner for helpful conversations.

%---------------------------------------------------------------------------------------------------------------------%
\section{Introductory Examples} \label{Sec:Examples}

We begin by recalling some of the basics of train tracks. See Penner and Harer for a thorough treatment of the topic \cite{PH}.

A train track on a surface $S$ is an embedded 1-complex $\tau$ satisfying the following three properties. 
\begin{enumerate}
\item Each edge (called a branch) is a smooth path with well-defined tangent vectors at the endpoints, and at any vertex (called a switch) the incident edges are mutually tangent. The tangent vector at the switch pointing toward the edge can have two possible directions which divides the ends of edges at the switch into two sets. The end of a branch of $\tau$ which is incident on a switch is called ``incoming" if the one-sided tangent vector of the branch agrees with the direction at the switch and ``outgoing" otherwise. 
\item Neither the set of ``incoming" nor the set of ``outgoing" branches are permitted to be empty. The valence of each switch is at least $3$, except for possibly one bivalent switch in a closed curve component. 
\item Finally, we require the components of $S \backslash \tau$ to have negative generalized Euler characteristic: for a surface $R$ whose boundary consists of smooth arcs meeting at cusps, define $\chi ' (R)$ to be the Euler characteristic $\chi (R)$ minus $1/2$ for every outward-pointing cusp (internal angle $0$), or plus $1/2$ for each inward pointing cusp (internal angle $2\pi$).
\end{enumerate}

A \textit{train route} is a non-degenerate smooth path in $\tau$. A train route traverses a switch only by passing from an incoming to an outgoing edge (or vice-versa). A \textit{transverse measure} on $\tau$ is a function $\mu$ which assigns to each branch a non-negative real number $\mu(b)$ which satisfies the switch condition: For any switch, the sums of $\mu$ over incoming and outgoing branches are equal. A train-track is \textit{recurrent} if there is a transverse measure which is positive on every branch, or equivalently if each branch is contained in a closed train route.

If $\sigma$ is a train track which is a subset of $\tau$ we write $\sigma < \tau$ and say $\sigma$ is a \textit{subtrack} of $\tau$. In this case we may also say that $\tau$ is an \textit{extension} of $\sigma$. If there is a homotopy of $S$ such that every train route on $\sigma$ is taken to a train route on $\tau$ we say $\sigma$ is \textit{carried} on $\tau$ and write $\sigma \prec \tau$.

Let $\alpha$ be a simple closed curve which intersects $\tau$. We say $\alpha$ intersects $\tau$ \textit{efficiently} if $\alpha \cup \tau$ has no bigon complementary regions. A track $\tau$ is \textit{transversely recurrent} if every branch of $\tau$ is crossed by some simple curve $\alpha$ intersecting $\tau$ transversely and efficiently. We call a track \textit{birecurrent} if it is both recurrent and transversely recurrent.

We call a train-track $\tau$ \textit{large} if every component of $S \backslash \tau$ is a polygon or a once-punctured polygon, and we call $\tau$ \textit{generic} if all switches are trivalent.\\

We now present explicit examples on how to use Theorems \ref{partitions} and \ref{modification}. Consider the six-times punctured sphere. By following Theorem\ref{partitions}, we now construct two pseudo-Anosov maps on $S_{0,6}$.

\begin{example}\label{example6punctures}
Consider the six-times punctured sphere and label the punctures of the sphere as introduced in Figure \ref{Fig:labellingsphere}. Up to spherical symmetry, there are two unique partitions of the six punctures so that the labels of the punctures are evenly spaced, namely $\mu = \{ \{0,3\}, \{1,4\}, \{2,5\}\}=\{\mu_1,\mu_2,\mu_3\}$ and $\bar{\mu} = \{ \{0,2,4\}, \{1,3,5\} \} = \{ \bar{\mu}_1, \bar{\mu}_2 \}$. Recall that we define the half-twist associated to puncture $j$ as the half-twist around the curve separating punctures $j$ and $j-1$. Therefore, these partitions can define the two maps, $\phi = D_{5}^2D_{2}^2D_{4}^2D_{1}^2D_{3}^2D_{0}^2=D_{\mu_3}^2D_{\mu_2}^2D_{\mu_1}^2$, and $\bar{\phi} = D_{5}^2D_{3}^2D_{1}^2D_{4}^2D_{2}^2D_{0}^2=D_{\bar{\mu}_2}^2D_{\bar{\mu}_1}^2$, respectively. We will prove that both maps are pseudo-Anosov.

We begin by proving that $\phi = D_{\mu_3}^2D_{\mu_2}^2D_{\mu_1}^2$ is a pseudo-Anosov map on $S_{0,6}$. In order to prove that $\phi$ is a pseudo-Anosov map, we find a train track $\tau$ on $S_{0,6}$ so that $\phi(\tau)$ is carried by $\tau$ and show that the matrix presentation of $\phi$ in the coordinates given by $\tau$ is a Perron-Frobenius matrix.

\begin{figure}[h]
\setlength{\unitlength}{0.01\linewidth}
\begin{picture}(65,20)
\put(0,0){\includegraphics{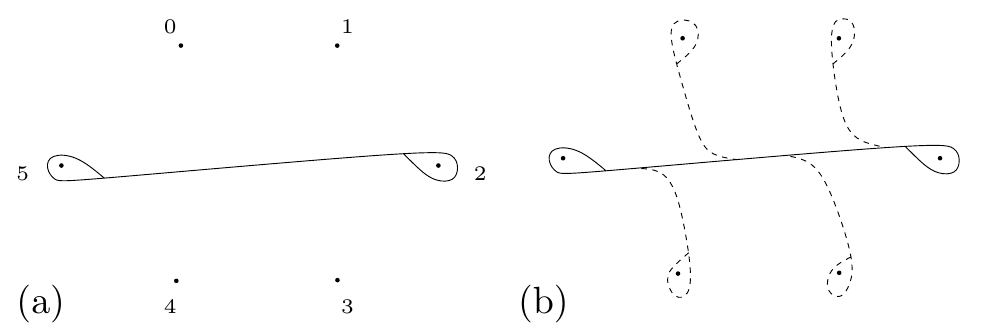}}
\end{picture}
\caption{Constructing the train track for the map $\phi$.}
\label{Fig:pA632construction} 
\end{figure}

We first describe how we construct the train track for the map $\phi$ based on the partition $\mu$. We notice that $\mu$ has three partitions containing two punctures each. Since the punctures in each partition are so that $\abs{i-j}\geq 2 \mod 6$, the twists associated to the punctures in each partition are disjoint. Since there are two twists in each partition and the partition is evenly spaced, the train track has rotational symmetry of order two. Therefore, we will construct a two-valent spine around the punctures labelled $2$ and $5$, as pictured in \ref{Fig:pA632construction} (a). Since there are three partitions, we will have two branches turning tangentially into each of the two nodes on the two-valent spine, where these branches will be turning left towards the spine, as pictured in \ref{Fig:pA632construction} (b).

\begin{figure}[h]
\setlength{\unitlength}{0.01\linewidth}
\begin{picture}(100,85)
\put(-6,-8){\includegraphics{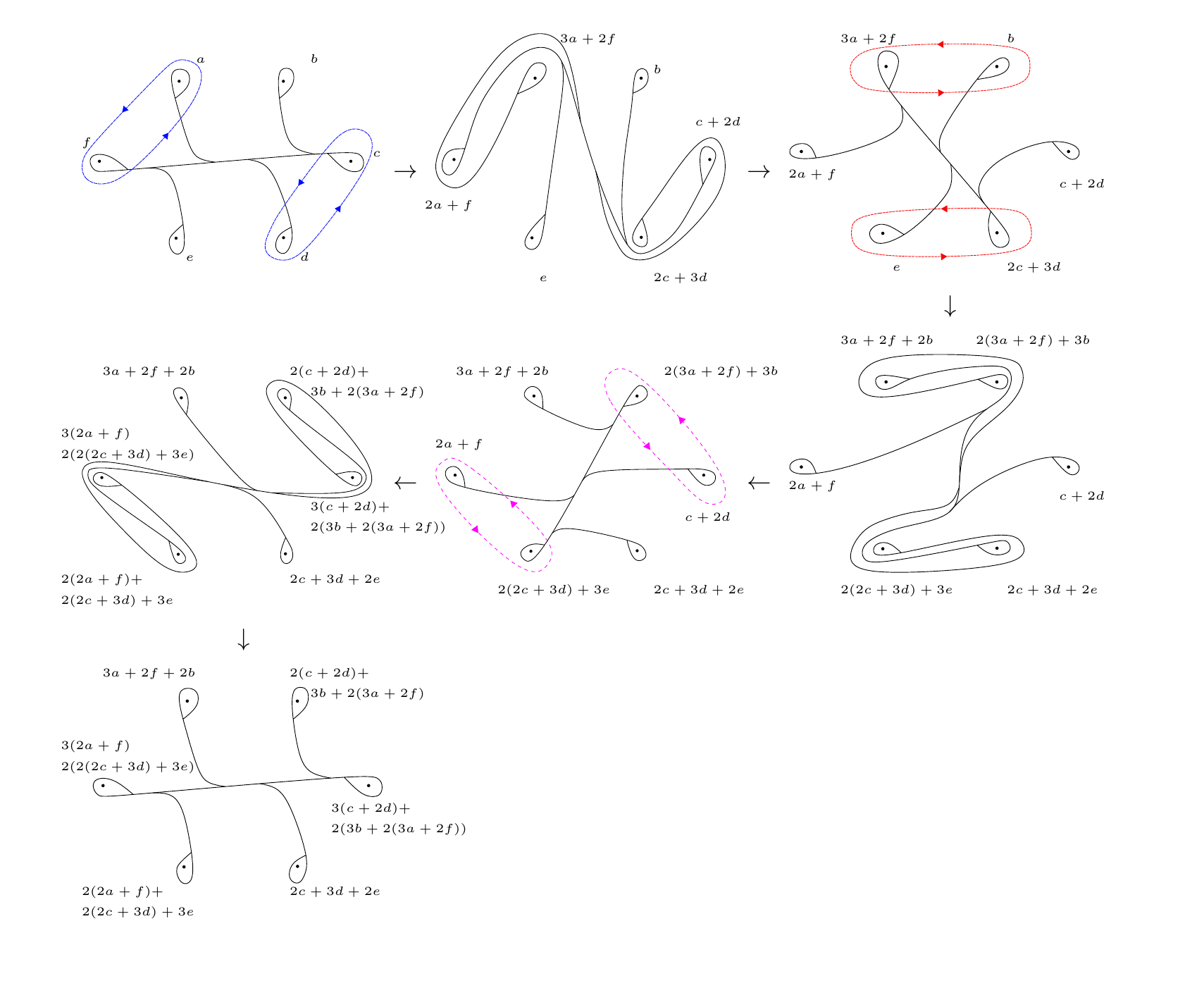}}
\end{picture}
\caption{The train track $\phi(\tau)$ is carried by $\tau$.}
\label{Fig:pA632} 
\end{figure}

The series of images in Figure \ref{Fig:pA632} depict the train track $\tau$ and its images under successive applications of the Dehn twists associated to $\phi$. These images prove that $\phi(\tau)$ is indeed carried by $\tau$, in fact, for every application of $D_{\mu_i}^{2}$, the train track $\tau$ rotates clockwise by $\frac{2\pi}{6}$. By keeping track of the weights on $\tau$, we calculate that the induced action on the space of weights on $\tau$ is given by the following matrix:
\[ A= \left( \begin{array}{cccccc}
3 & 2 & 0 & 0 & 0 & 2 \\
6 & 3 & 2 & 4 & 0 & 4 \\
12 & 6 & 3 & 6 & 0 & 8 \\
0 & 0 & 2 & 3 & 2 & 0 \\
4 & 0 & 4 & 6 & 3 & 2 \\
6 & 0 & 8 & 12 & 6 & 3 \end{array} \right)\] 
Note that the space of admissible weights on $\tau$ is the subset of $\RR^6$ given by positive real numbers $a, b, c, d, e$ and $f$ such that $a+b+f = c+d+e$. The linear map described above preserves this subset. The square of the matrix $A$ is strictly 
positive, which implies that the matrix is a Perron-Frobenius matrix. In fact, the top eigenvalue is $9+4\sqrt{5}$ which is associated to a unique irrational measured lamination $F$ carried by $\tau$ that is fixed by $\phi$. Lastly, since the train track $\tau$ is large, generic, and birecurrent, we are able to apply Lemma \ref{nesting} from Section \ref{Sec:Main} which finishes the proof that this map is a pseudo-Anosov.

Notice that we can perform each of the half twists to any power and still have the exact same train track constructed above. However, the labels associated to the branches affected by the change in the number of twists will subsequently increase or decrease in value according to how many twists are applied to each curve. Since all the twists are positive, we still have that all values in the resulting matrix will be positive and will be a Perron-Frobenius. An application of Lemma \ref{nesting} from Section \ref{Sec:Main} will give our desired result.

We will now show that $\bar{\phi} = D_{\bar{\mu}_2}^2D_{\bar{\mu}_1}^2$ is a pseudo-Anosov on $S_{0,6}$, which follows a similar argument to the above.

\begin{figure}[h]
\setlength{\unitlength}{0.01\linewidth}
\begin{picture}(65,26)
\put(0,0){\includegraphics{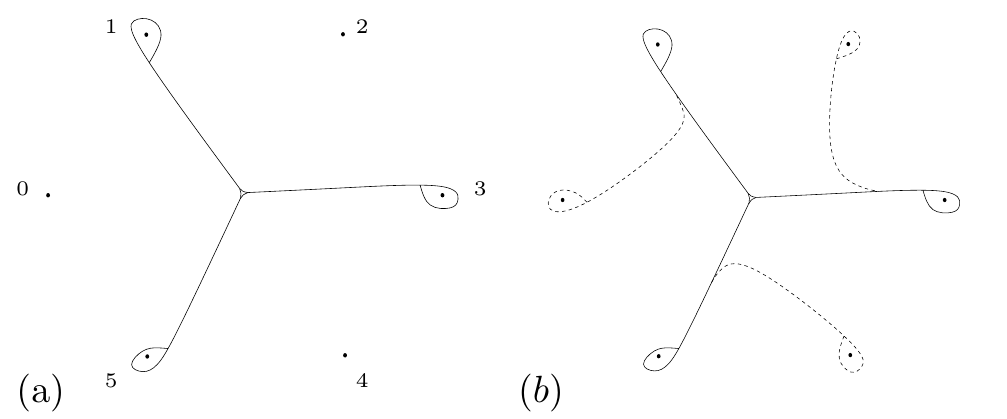}}
\end{picture} 
\caption{Constructing the train track for the map $\bar{\phi}$.}
\label{Fig:pA623construction} 
\end{figure}

We again analyze the partition $\bar{\mu}$ as it determines the construction of our train track $\bar{\tau}$. Notice that $\bar{\mu}$ has two partitions containing three twists each. Since there are three twists in each partition and the partition is evenly spaced, the train track has rotational symmetry of order three. Therefore, we will construct a three-valent spine around the punctures labelled $1$, $3$ and $5$, as pictured in \ref{Fig:pA623construction} (a). Since there are two partitions, we will have one branch turning tangentially towards each of the three nodes on the three-valent spine, where these branches will be turning left towards the spine, as pictured in \ref{Fig:pA623construction} (b).

\begin{figure}[h]
\setlength{\unitlength}{0.01\linewidth}
\begin{picture}(100,64)
\put(-9,-8){\includegraphics{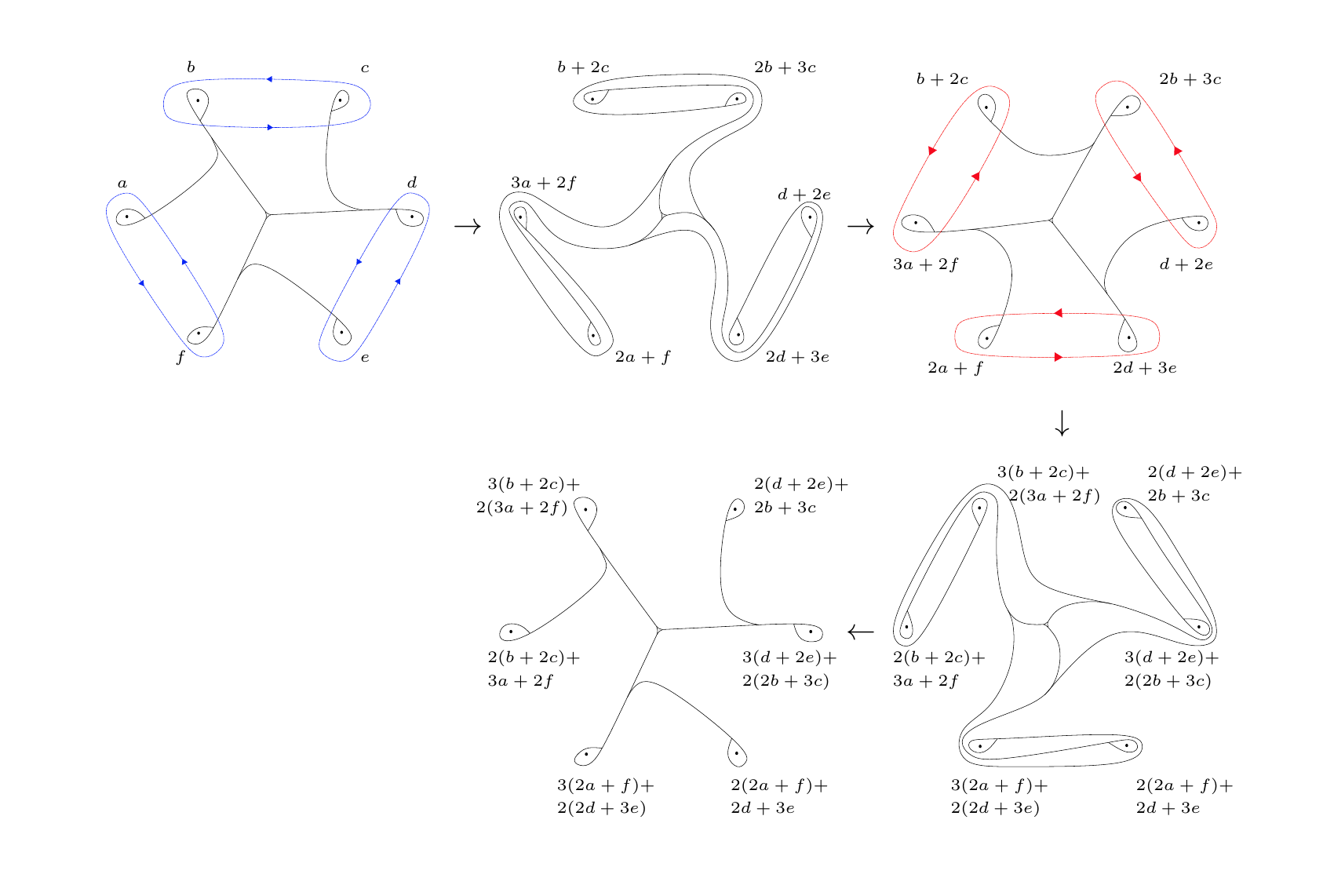}}
\end{picture} 
\caption{The train track $\bar{\phi}(\bar{\tau})$ is carried by $\bar{\tau}$.}
\label{Fig:pA623} 
\end{figure}

The series of images in Figure \ref{Fig:pA623} depict the train track $\bar{\tau}$ and its images under successive applications of the Dehn twists associated to $\bar{\phi}$. These images prove that $\bar{\phi}(\bar{\tau})$ is indeed carried by $\bar{\tau}$. We again notice that for every application of $D_{\bar{\mu}_i}^{2}$, the train track $\bar{\tau}$ rotates clockwise by $\frac{2\pi}{6}$. By keeping track of the weights on $\bar{\tau}$, we calculate that the induced action on the space of weights on $\tau$ is given by the following matrix:
\[ B= \left( \begin{array}{cccccc}
3 & 2 & 4 & 0 & 0 & 2 \\
6 & 3 & 6 & 0 & 0 & 4 \\
0 & 2 & 3 & 2 & 4 & 0 \\
0 & 4 & 6 & 3 & 6 & 0 \\
4 & 0 & 0 & 2 & 3 & 2 \\
6 & 0 & 0 & 4 & 6 & 3 \end{array} \right)\] 
The space of admissible weights on $\bar{\tau}$ is the subset of $\RR^6$ given
by positive real numbers $a, b, c, d, e$ and $f$ such that $b-a, d-c,$ and $f-e$ are all positive and satify the triangle inequalities. The linear 
map described above preserves this subset. The square of the matrix $B$ is strictly 
positive, which implies that the matrix is a Perron-Frobenius matrix. The top eigenvalue is $7+4\sqrt{3}$, which is associated to a unique irrational measured lamination $F$ carried by $\bar{\tau}$ that is fixed by $\bar{\phi}$. As the train track is large, generic, and birecurrent, we may apply Lemma \ref{nesting} to finish the proof that this map is a pseudo-Anosov.
\end{example}

We will now modify the pseudo-Anosov maps from Example \ref{example6punctures} in order to find two pseudo-Anosov maps on the seven-times punctured sphere. To do this, we will apply Theorem \ref{modification} once to each of the maps found in \ref{example6punctures}. For each of these maps, we note that we can apply the modification more than once to obtain additional pseudo-Anosov maps defined on spheres with more punctures.

\begin{example}\label{example7punctures}
We now consider the seven-times punctured sphere and with the labelling introduced in Theorem \ref{partitions}. We will consider two different ``partitions" which are modifications of the partitions from Example \ref{example6punctures}. The two ``partitions" we obtain after applying Theorem \ref{modification} are $\mu ' = \{ \{0,4\}, \{1,5\}, \{2,6\}, \{3\}\}=\{\mu_1 ' , \mu_2 ' ,\mu_3 ' , \mu_4 ' \}$ and $\bar{\mu} ' = \{ \{0,3,5\}, \{1,4,6\}, \{2\} \} = \{ \bar{\mu}_1 ' , \bar{\mu}_2 ' , \bar{ \mu_3} ' \}$.

\begin{figure}[h]
\setlength{\unitlength}{0.01\linewidth}
\begin{picture}(65,29)
\put(0,0){\includegraphics{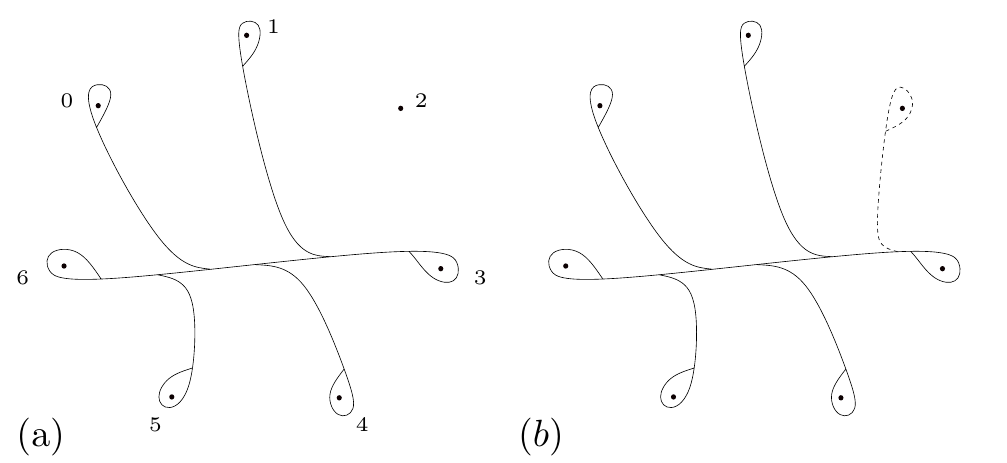}}
\end{picture} 
\caption{Constructing the train track for the map $\psi$.}
\label{Fig:pA732construction} 
\end{figure}

We will begin by proving that $\phi ' = D_{\mu_4 '}^2 D_{\mu_3 '}^2 D_{\mu_2 '}^2 D_{\mu_1 '}^2$ is a pseudo-Anosov on $S_{0,7}$. First, we will describe how to construct the train track for the map $\phi '$, denoted $\tau '$, from the train track $\tau$ associated to the map $\phi$ from the previous example. Consider the train track $\tau$ and place an extra puncture between the punctures labelled $1$ and $2$ in the previous example, and then relabel the punctures so that the labelling is as in Theorem \ref{partitions} (see Figure \ref{Fig:pA732construction} (a)). Therefore, we have a train track without a branch around the puncture labelled $2$, but the rest of the train track is as in Example \ref{example6punctures} (up to relabelling). We construct a branch around the puncture labelled $2$ which will turn tangentially towards the two valent spine, turning left towards the puncture labelled $3$ (see Figure \ref{Fig:pA732construction} (b)).

\begin{figure}[h]
\setlength{\unitlength}{0.01\linewidth}
\begin{picture}(100,110)
\put(-9,-5){\includegraphics{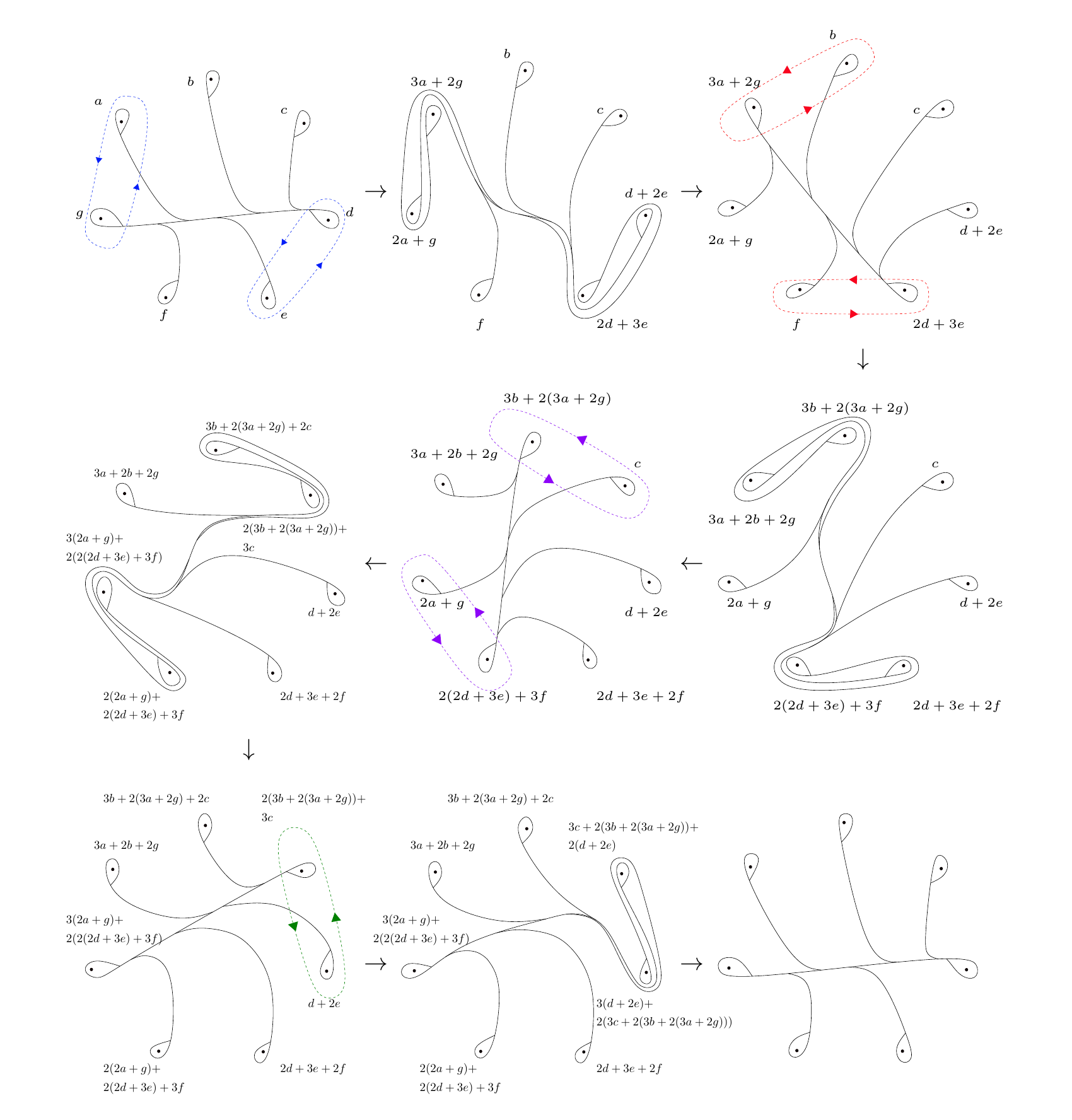}}
\end{picture} 
\caption{The train track $\phi '(\tau ')$ is carried by $\tau '$.}
\label{Fig:pA732} 
\end{figure}

The series of images in Figure \ref{Fig:pA732} depict the train track $\tau '$ and its images under successive applications of the Dehn twists associated to $\phi '$. These images prove that $\phi ' (\tau ')$ is carried by $\tau '$. By keeping track of the weights on $\tau '$, we calculate that the induced action on the space of weights on $\tau '$ is given by the following matrix:
\[ C= \left( \begin{array}{ccccccc}
3 & 2 & 0 & 0 & 0 & 0 & 2 \\
6 & 3 & 2 & 0 & 0 & 0 & 4 \\
12 & 6 & 3 & 2 & 0 & 0 & 8\\
24 & 12 & 6 & 3 & 6 & 0 & 16 \\
0 & 0 & 0 & 2 & 3 & 2 & 0 \\
4 & 0 & 0 & 4 & 6 & 3 & 2 \\
6 & 0 & 0 & 8 & 12 & 6 & 3 \end{array} \right)\] 
The space of admissible weights on $\tau '$ is the subset of $\mathbb{R}^7$ given by the positive real numbers $a, b, c, d, e, f,$ and $g$ such that $a+b+d+f=c+e+g$. The linear map described above preserves this subset. The square of the matrix $C$ is strictly positive, which implies that the matrix is a Perron-Frobenius. Additionally, the top eigenvalue is approximately $22.08646$ which is associated to a unique irrational measured lamination $F$ carried by $\tau '$ which is fixed by $\phi '$. As the train track is large, generic, and birecurrent, we may apply Lemma \ref{nesting} to finish the proof that this map is a pseudo-Anosov.

\begin{figure}[h]
\setlength{\unitlength}{0.01\linewidth}
\begin{picture}(65,29)
\put(0,0){\includegraphics{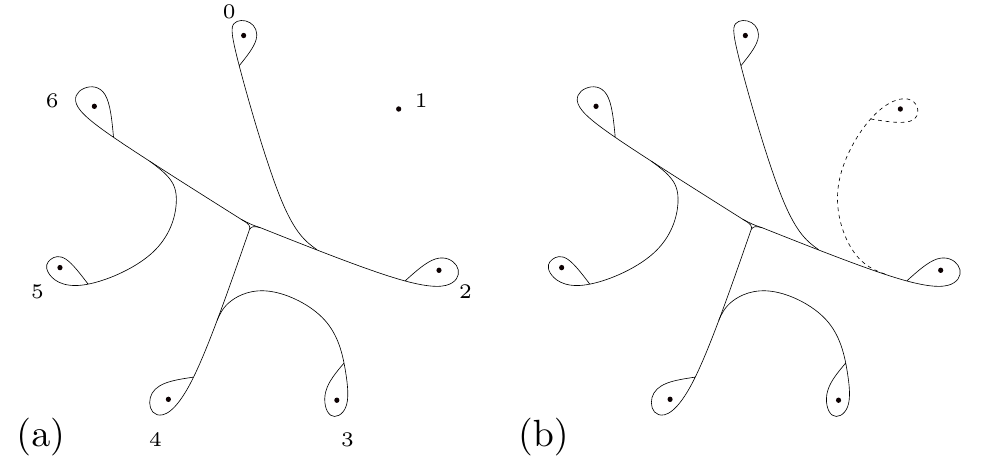}}
\end{picture} 
\caption{Constructing the train track for the map $\bar{\phi '}$.}
\label{Fig:pA723construction} 
\end{figure}

We now show that the homeomorphism $\bar{\tau '} = D_{\bar{\mu_3} '}^2 D_{\bar{\mu_2} '}^2 D_{\bar{\mu_1} '}^2$ is a pseudo-Anosov map. To construct the train track, we will consider the train track $\bar{\tau}$ associated to the map $\bar{\phi}$ from the previous example. Consider the train track $\bar{\tau}$ and place an extra puncture between the punctures labelled $0$ and $1$ in the previous example, and then relabel the punctures so that the labelling is as in Theorem \ref{partitions} (see Figure \ref{Fig:pA723construction} (a)). Therefore, we have a train track on $S_{0,7}$ which does not have a branch around the puncture labelled $1$, and the rest of the train track is as found in Example \ref{example6punctures} (up to relabelling). We then construct a branch around the puncture labelled $1$ which will turn tangentially into the three valent spine, turning left towards the puncture labelled $2$ (see Figure \ref{Fig:pA723construction} (b)).

\begin{figure}[h]
\setlength{\unitlength}{0.01\linewidth}
\begin{picture}(100,110)
\put(-9,-5){\includegraphics{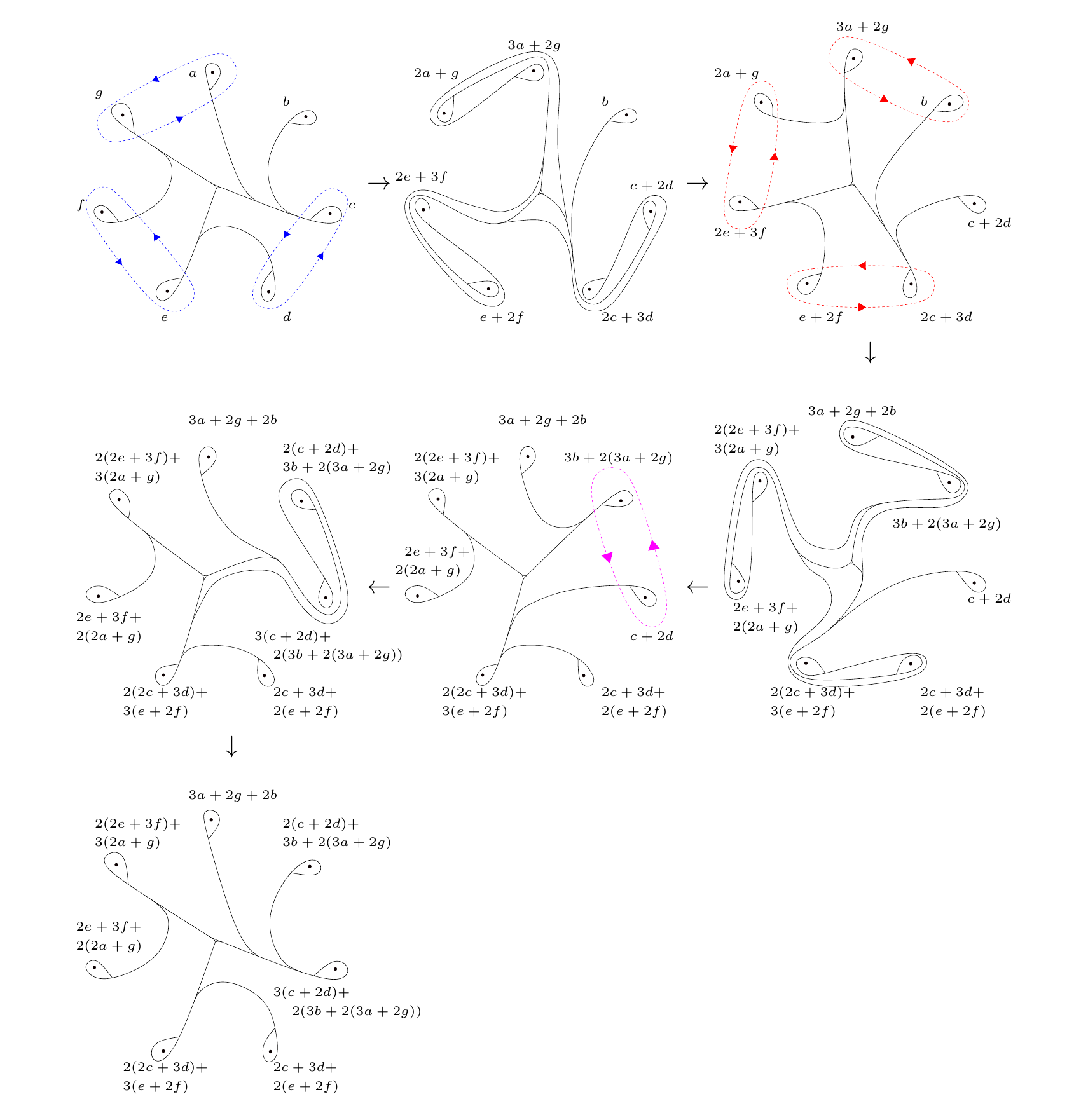}}
\end{picture}
\caption{The train track $\bar{\phi} '(\bar{\tau} ')$ is carried by $\bar{\tau} '$.}
\label{Fig:pA723} 
\end{figure}

The series of images in Figure \ref{Fig:pA723} depict the train track $\bar{\tau '}$ and its images under successive applications of the Dehn twists associated to $\bar{\phi '}$.

Figure \ref{Fig:pA723} shows that $\bar{\phi '}(\bar{\tau '})$ is indeed carried by $\bar{\tau '}$. By keeping track of the weights on $\bar{\tau '}$, we calculate that the induced action on the space of weights on $\bar{\tau '}$ is given by the following matrix:
\[ D= \left( \begin{array}{ccccccc}
3 & 2 & 0 & 0 & 0 & 0 & 2 \\
6 & 3 & 2 & 4 & 0 & 0 & 4 \\
12 & 6 & 3 & 6 & 0 & 0 & 8\\
0 & 0 & 2 & 3 & 2 & 4 & 0 \\
0 & 0 & 4 & 6 & 3 & 6 & 0 \\
4 & 0 & 0 & 0 & 2 & 3 & 2 \\
6 & 0 & 0 & 0 & 4 & 12 & 3 \end{array} \right)\] 
The space of admissible weights on $\bar{\tau '}$ is the subset of $\mathbb{R}^7$ given by the positive real numbers $a, b, c, d, e, f,$ and $g$ such that $c-b-a, e-d,$ and $g-f$ are all positive and satisfy the triangle inequalities. The linear map described above preserves this subset. The square of the matrix $D$ is strictly positive, which implies that the matrix is a Perron-Frobenius. The top eigenvalue of this matrix is 
\[
5+\frac{1}{3}\sqrt[3]{2916-12\sqrt{93}}+\left( \frac{2}{3}\right)^{2/3}\sqrt[3]{243+\sqrt{93}},
\] 
which is associated to a unique irrational measured lamination $F$ carried by $\bar{\tau} '$ which is fixed by $\bar{\phi} '$. As the train track is large, generic, and birecurrent, we may apply Lemma \ref{nesting} to finish the proof that this map is a pseudo-Anosov.
\end{example}

%---------------------------------------------------------------------------------------------------------------------%

\section{Construction on n-times Punctured Spheres} \label{Sec:Main}

For a recurrent train-track $\tau$, let $P(\tau)$ denote the polyhedron of measures supported on $\tau$. Notice that $P(\tau)$ is preserved by scaling. By $int(P(\tau))$ we denote the set of weights on $\tau$ which are positive on every branch. We say that $\sigma$ \textit{fills} $\tau$ if $\sigma \prec \tau$ and $int (P(\sigma)) \subseteq int (P(\tau))$. Similarly, a curve $\alpha$ fills $\tau$ if $\alpha \prec \tau$ and $\alpha$ traverses every branch of $\tau$.

Let $\sigma$ be a large track. A \textit{diagonal extension} of $\sigma$ is a track $\kappa$ such that $\sigma < \kappa$ and every branch of $\kappa \backslash \sigma$ is a \textit{diagonal} of $\sigma$ ie. its endpoints terminate in the corner of a complementary region of $\sigma$. Let $E(\sigma)$ denote the set of all recurrent diagonal extensions of $\sigma$. Note that it is a finite set, and let $PE(\sigma)$ denote $\bigcup_{\kappa \in E(\sigma)}P(\kappa)$. Let $int(PE(\sigma))$ denote the set of measures $\mu \in PE(\sigma)$ which are positive on every branch of $\sigma$.

We now prove the following lemma, which is similar to the Nesting Lemma from \cite{MM1}. This lemma is the final step of the proof to show that the maps constructed are pseudo-Anosov.

\begin{lemma}\label{nesting}
Let $\tau$ be a large, generic, birecurrent train track. Let $\phi \from S \to S$ be a map which preserves $\tau$. If the matrix associated to $\tau$ is a Perron-Frobenius matrix, then $\phi$ is a pseudo-Anosov map.
\end{lemma}

\begin{proof}
Consider $\mu \in int(P(\phi (\tau)))$. Therefore, $\mu$ is a measure which is positive on every branch of $\phi (\tau)$. Since $\phi(\tau)$ preserves $\tau$, this implies that $\mu$ is also a measure which is positive on any branch of $\tau$. Therefore, we have that $int(P(\phi(\tau))) \subseteq int(P(\tau))$, which implies that $\phi(\tau)$ fills $\tau$.

Since $\phi(\tau)$ fills $\tau$, we may follow lines 15-37 of the proof of Theorem 4.6 from \cite{MM1} to show that there exists some $k \in \mathbb{N}$ such that $\phi^{k} (PE(\tau)) \subset int(PE(\tau))$. Similarly, we have that $\phi^{jk}(\tau)$ fills $\phi^{(j-1)k}(\tau)$, which shows that
\begin{equation}\label{subsetPEint}
PE(\phi^{jk}(\tau)) \subset int(PE(\phi^{(j-1)k}(\tau)))
\end{equation}
for any $j$.

By way of contradiction, suppose that $\phi$ is not a pseudo-Anosov map. Then there exists a curve $\alpha$ on the surface $S$ disjoint from $\tau$ such that there exists an $m \in \mathbb{N}$ such that $\phi^{m} (\alpha) = \alpha$.

Since there are elements of $PE(\tau)$ which are not in $\phi^{k}(PE(\tau))$, it is possible to find $\gamma \in C_{0}(S)$ such that $\gamma \notin PE(\tau)$ and $\phi^{k}(\gamma) \in PE(\tau)$. Then $\phi^{2k}(\gamma) \in int(PE(\tau))$, which implies that $d_{C_{0}(S)}(\gamma, \phi^{2k}(\gamma)) \geq 1$ by Lemma 3.4 of \cite{MM1}.

Since $\phi^{jk}(\gamma) \in PE(\phi^{(j-1)k}(\tau))$ for $j \geq 1$, we use Equation \ref{subsetPEint} to find
\begin{equation*}
\begin{aligned}
\phi^{3k}(\gamma) &\in PE(\phi^{2k}(\tau)) \subset int(PE(\phi^{k}(\tau)))\\
\phi^{3k}(\gamma) &\in PE(\phi^{k}(\tau)) \subset int(PE(\tau))\\
\phi^{3k}(\gamma) &\in PE(\tau).
\end{aligned}
\end{equation*}
By Lemma 3.4 of \cite{MM1}, we have that for any $k$, $\mathcal{N}_1(int(PE(\phi^{k}(\tau)))) \subset PE(\phi^{k}(\tau))$. Therefore, we find that
\[
\phi^{3k}(\gamma) \in PE(\phi^{2k}(\tau)) \subset \mathcal{N}_1(int(PE(\phi^{k}(\tau)))) \subset PE(\phi^{k}(\tau)) \subset \mathcal{N}_1(int(PE(\tau))) \subset PE(\tau).
\]
Therefore, since $\gamma \notin PE(\tau)$, we have that $d_{C_{0}(S)}(\gamma, \phi^{3k}(\gamma)) \geq 2$.

One argues inductively to show that $d_{C_{0}(S)}(\gamma, \phi^{jk}(\gamma )) \geq j-1$, which implies that as $j \rightarrow \infty$, $d_{C_{0}(S)}(\gamma , \phi^{jk}(\gamma)) \rightarrow \infty$. Since
\begin{equation}
\begin{aligned}
d_{C_{0}(S)}(\alpha, \phi^{jk}(\alpha)) &\geq d_{C_{0}(S)}(\gamma , \phi^{jk}(\gamma )) - d_{C_{0}(S)}(\alpha, \gamma ) - d_{C_{0}(S)}(h^{jk}(\alpha), h^{jk}(\gamma ))\\
&= d_{C_{0}(S)}(\gamma , \phi^{jk}(\gamma)) - 2 d_{C_{0}(S)}(\alpha, \gamma )
\end{aligned}
\end{equation}
we have $d_{C_{0}(S)}(\alpha, \phi^{jk}(\alpha)) \rightarrow \infty$, which contradicts that there exists some $m$ such that $\phi^{m} (\alpha) = \alpha$. Therefore $\phi$ is a pseudo-Anosov map.
\end{proof}

Now that we have established Lemma \ref{nesting}, we are able to prove Theorem \ref{partitions}.

\begin{proof}[Proof of Theorem \ref{partitions}]
Fix some value of $n \in \mathbb{N}$, and consider the surface $S_{0,n}$. Fix $k>1$, $k \in \mathbb{N}$, and fix a partition $\mu = \{ \mu_1, \ldots, \mu_k \}$ of the $n$ punctures of $S_{0,n}$ such that $\rho(\mu_{i-1}) = \rho(\mu_{(i \mod k)})$. We prove that 
\[
\phi = \prod_{i=1}^{k} D_{\mu_i}^{q_i} = D_{\mu_k}^{q_k} \ldots D_{\mu_2}^{q_2} D_{\mu_1}^{q_1}
\]
is a pseudo-Anosov mapping class.

We first construct the train track $\tau$ so that $\phi(\tau)$ is carried by $\tau$. Consider the partition $\mu = \{ {\mu_1}, \ldots, {\mu_k} \}$. Construct a $k$-valent spine by having a branch loop around the highest labelled puncture in each of the sets $\mu_i$, with each of these branches meeting in the center where they are smoothly connected by a $k$-gon, such as in Figures \ref{Fig:pA632construction} (a) and \ref{Fig:pA623construction} (a). For the remaining labelled punctures in $\mu_i$, loop a branch around each puncture and have this branch will turn left towards the $k$-valent spine meeting the branch of the spine whose label is next in the ordering, such as in Figures \ref{Fig:pA632construction} (b) and \ref{Fig:pA623construction} (b).

For each $k$, $D_{\mu_i}^{q_i}$ acts locally the same. In particular, each half-twist in $D_{\mu_i}^{q_i}$ involves a branch located around puncture $b$ on the $k$-valent spine and the branch located around puncture $b'$ which is directly next to puncture $b$ in the clockwise direction. As we consider a right half-twist to be positive, we notice that the branch at puncture $b$ will begin to turn into the branch at puncture $b'$, see Figures \ref{Fig:pA632} and \ref{Fig:pA623} for examples. Therefore, after the twist, the branch around puncture $b'$ is now on the $k$-valent spine, and the branch around puncture $b$ is directly next to the branch at puncture $b'$ in the counter clockwise direction. Branches which are neither on the $k$-valent spine nor directly clockwise to the $k$-valent spine are unaffected by $D_{\mu_i}^{q_i}$. Thus, after each application of $D_{\mu_i}^{q_i}$, the train track rotates clockwise by $\frac{2\pi}{n}$. Since $\tau$ has a rotational symmetry of order $k$, we notice that $\phi(\tau)$ is carried by $\tau$.

Let $M_{\tau}$ denote the matrix representing the induced action of the space of weights on $\tau$. To prove that $M_{\tau}$ is Perron-Frobenius, fix an initial weight on each branch. For each application of $D_{\mu_i}^{q_i}$, the labels on and directly next in the clockwise direction to the $k$-valent spine will become a linear combination of the labels associated to these two branches. In particular, let $w$ be the weight of a branch on the $k$-valent spine, and let $w'$ be the weight of the branch directly next in the clockwise direction to the branch on the $k$-valent spine. After applying $l$ half-twists, we see that the weight of branch $w$ is $lw'+(l-1)w$ and the weight of branch $w'$ is $(l+1)w'+lw$. Since $\tau$ rotates clockwise by $\frac{2\pi}{n}$ after each application of $D_{\mu_i}^{q_i}$, we know that after $k$ applications of $\phi$ each branch will be a linear combination of the initial weights from each of the branches where the constants of this linear combination will be strictly positive integers. Equivalently, this implies that each entry in $M_{\tau}^{k}$ is a strictly positive integer value. This implies that the matrix $M_{\tau}$ is Perron-Frobenius.

To finish the proof, note that each of the train tracks which were constructed above are large, generic, and birecurrent. Therefore, we apply Lemma \ref{nesting} which completes the proof that $\phi$ is a pseudo-Anosov mapping class.
\end{proof}

We now provide a proof for Theorem \ref{modification}.

\begin{proof}[Proof of Theorem \ref{modification}]
Fix some value of $n \in \mathbb{N}$, some $k>1$, $k \in \mathbb{N}$, and a partition $\mu = \{ \mu_1, \ldots, \mu_k \}$ of the $n$ punctures of $S_{0,n}$ such that $\rho(\mu_{i-1}) = \rho(\mu_{(i \mod k)})$. We perform the modification outlined in the statement of the theorem to obtain a partition, $\mu ' = \{ \mu_1 ', \ldots, \mu_k ', \mu_{k+1} ' \}$, on the $(n+1)$-times punctured sphere, $S_{0,n+1}$, which defines the map 
\[
\phi '= \prod_{i=1}^{k+1} D_{\mu_i}^{q_i '} = D_{\mu_{k+1} '}^{q_{k+1} '} D_{\mu_k '}^{q_k '} \ldots D_{\mu_2 '}^{q_2 '} D_{\mu_1 '}^{q_1 '}.
\]
We prove that $\phi '$ is a pseudo-Anosov mapping class.

We first construct the train track $\tau '$ so that $\phi ' (\tau ')$ is carried by $\tau '$. We begin by considering the train track $\tau$ associated to the map $\phi = \prod_{i=1}^{k} D_{\mu_i}^{q_i} $ defined by the partition $\mu$. We then add in a new puncture onto the sphere between punctures $k$ and $k+1$, and relabel the punctures, see Figures \ref{Fig:pA732construction} (a) and \ref{Fig:pA723construction} (a) for examples. Add a branch from puncture $k+1$ so that it turns tangentially into the $k$-valent spine meeting the same branch on the spine as the branches associated to punctures $1, \ldots, j$, see Figures \ref{Fig:pA732construction} (b) and \ref{Fig:pA723construction} (b) for examples. We denote this modified train track by $\tau '$.

To show that $\phi '(\tau ')$ is carried by $\tau '$, we notice that by the same reasoning in the proof of Theorem \ref{partitions} that for each $1 \leq i < k+1$, the application of $D_{\mu_i '}^{q_i '}$ will rotate the train track clockwise by $\frac{2\pi}{n}$. After the first $k$ applications of $D_{\mu_i '}^{q_i '}$, we have rotated the train-track by $\frac{2\pi k}{n}$, which is not quite $\tau '$. By applying the final twist $D_{\mu_{k +1} '}^{q_{k+1}'}$, we find $\phi ' (\tau ') = \tau '$ and thus $\phi ' (\tau ')$ is carried by $\tau '$. See Figures \ref{Fig:pA732} and \ref{Fig:pA723} for examples.

For the same reasoning as in the proof of Theorem \ref{partitions}, the matrix representing the induced action on the space of weights on $\tau'$ will be a Perron-Frobenius. To finish the proof, we note that each of the train tracks that we have constructed are large, generic, and birecurrent. Therefore, we are able to apply Lemma \ref{nesting} which completes the proof that the map is a pseudo-Anosov mapping class.
\end{proof}

%---------------------------------------------------------------------------------------------------------------------%

\section{Modifications of Construction} \label{Sec:Modifications}

By considering the constructions described in Section \ref{Sec:Main}, we notice that there are additional modifications one can make to the construction to obtain more pseudo-Anosov mapping classes.\\

The first modification we notice is that it is possible to apply the modification outlined in Theorem \ref{modification} more than once to construct additional pseudo-Anosov maps. In fact, we may continue to apply the modification step ad Infinitum to continue to find maps on punctured spheres.

Furthermore, one can apply the modification step from Theorem \ref{modification} in a slightly different manner in order to construct other pseudo-Anosov maps. Notice that the first application of Theorem \ref{modification} can allow a puncture to be placed between any of the punctures located on the $k$-valent spine and the puncture directly counter clockwise in the labelling. We may modify the initial partition by including a set which contains up to $k-1$ of the $k$ punctures which may be placed between a puncture on the spine and the element directly counter clockwise in the labelling. One may not include all $k$ of the punctures as this map will be equivalent to a map found using Theorem \ref{partitions}. Up to relabelling the punctures, we will have defined a pseudo-Anosov map. The proof to show this is a pseudo-Anosov map follows the proof of Theorem \ref{modification} but with a modified train track where an additional branch turns tangentially into the branch of the $k$-valent spine directly clockwise in the labelling of the punctures.

Both modifications described above can be made ad Infinium to any of the maps described in Theorem \ref{partitions}, as long as any of the new sets created never contain the same number of elements as any of the sets from the initial partition.\\

To obtain a third modification, we consider a map $\phi$ from any of the possible maps found though Theorem \ref{modification} or any of the modifications outlined above. It is possible to find an additional pseudo-Anosov map which has the same train track as $\phi$. Since the train track rotates by $\frac{2\pi}{n}$ for the first $k$ applications of $D_{\mu_i '}^{q_i '}$, we are able to define a map that will continue to rotate the train track by $\frac{2\pi}{n}$ in place of doing the final twist(s) $D_{\mu_{k+1}'}^{q_{k+1} '}$.

For example, consider the first map from Example \ref{example7punctures}. After applying $D_{\mu_3 '}^2 D_{\mu_2'}^2 D_{\mu_1 '}^2$, where $\mu ' = \{ \{1,5\}, \{2,6\}, \{3,7\}\}=\{\mu_1 ' ,\mu_2 ' ,\mu_3 ' \}$, the train track has rotated by $\frac{6 \pi}{7}$. We can apply the rotations associated to punctures $4$ and $1$ next, then around puncture $5$ and $2$, around punctures $6$ and $3$, and finally around punctures $7$ and $4$, which will have rotated our train track by a full rotation. In other words, you will obtain a new ``partition" containing the sets $\mu_i '' = \{i,i+\lceil \frac{7}{3} \rceil \}$ for $1 \leq i \leq 7$. More precisely, we obtain the following additional construction:

\begin{theorem}
Consider the surface $S_{0,n}$. Consider one of the maps from Theorem \ref{partitions}, in particular, consider a partition of the $n$ punctures into $1<k<n$ sets $\{ \mu_i \}_{i=1}^{k}$ such that the partition is evenly spaced. Apply any number of applications of Theorem \ref{modification} to obtain a new ``partition", $\mu ' = \{ \mu_1 ', \ldots, \mu_k ', \mu_{k+1}', \ldots, \mu_{k+l} ' \}$, where $|\mu_{k+j}| < |\mu_{1}|$ for all $1 \leq j \leq l$ which defines a map on the $p$-sphere, where $p = \sum_{i=1}^{k+l} |\mu_{i}|$. Consider the train track $\tau '$ associated to this map. Define the partition $\mu ''$ to be the partition containing the sets $\mu_i '' = \{ i, i+\lceil \frac{p}{k} \rceil\, \ldots, i + (|\mu_{1}|-1) \lceil \frac{p}{k} \rceil \}$, where $1 \leq i \leq p$. $\mu ''$ defines the pseudo-Anosov mapping class
\[
\phi '= \prod_{i=1}^{p} D_{\mu_i ''}^{q_i ''} = D_{\mu_{p} ''}^{q_{p} ''} \ldots D_{\mu_2 ''}^{q_2 ''} D_{\mu_1 ''}^{q_{1} ''}
\]
on $S_{0,p}$, where $q_j '' = \{q_{j_1} '', \ldots, q_{j_l} '' \}$ is the set of powers associated to each $\mu_i ''$.
\end{theorem}

%---------------------------------------------------------------------------------------------------------------------%

\section{Construction on Surfaces of Higher Genus} \label{Sec:HigherGenus}

We can lift the constructed pseudo-Anosov mapping classes on $2g+2$-punctured spheres to pseudo-Anosov mapping classes on surfaces of genus $g>0$ through a branched cover. For an overview of Birman-Hilden theory, see \cite{MW}.

\begin{figure}[h]
\setlength{\unitlength}{0.01\linewidth}
\begin{picture}(80,82)
\put(-2,-5){\includegraphics{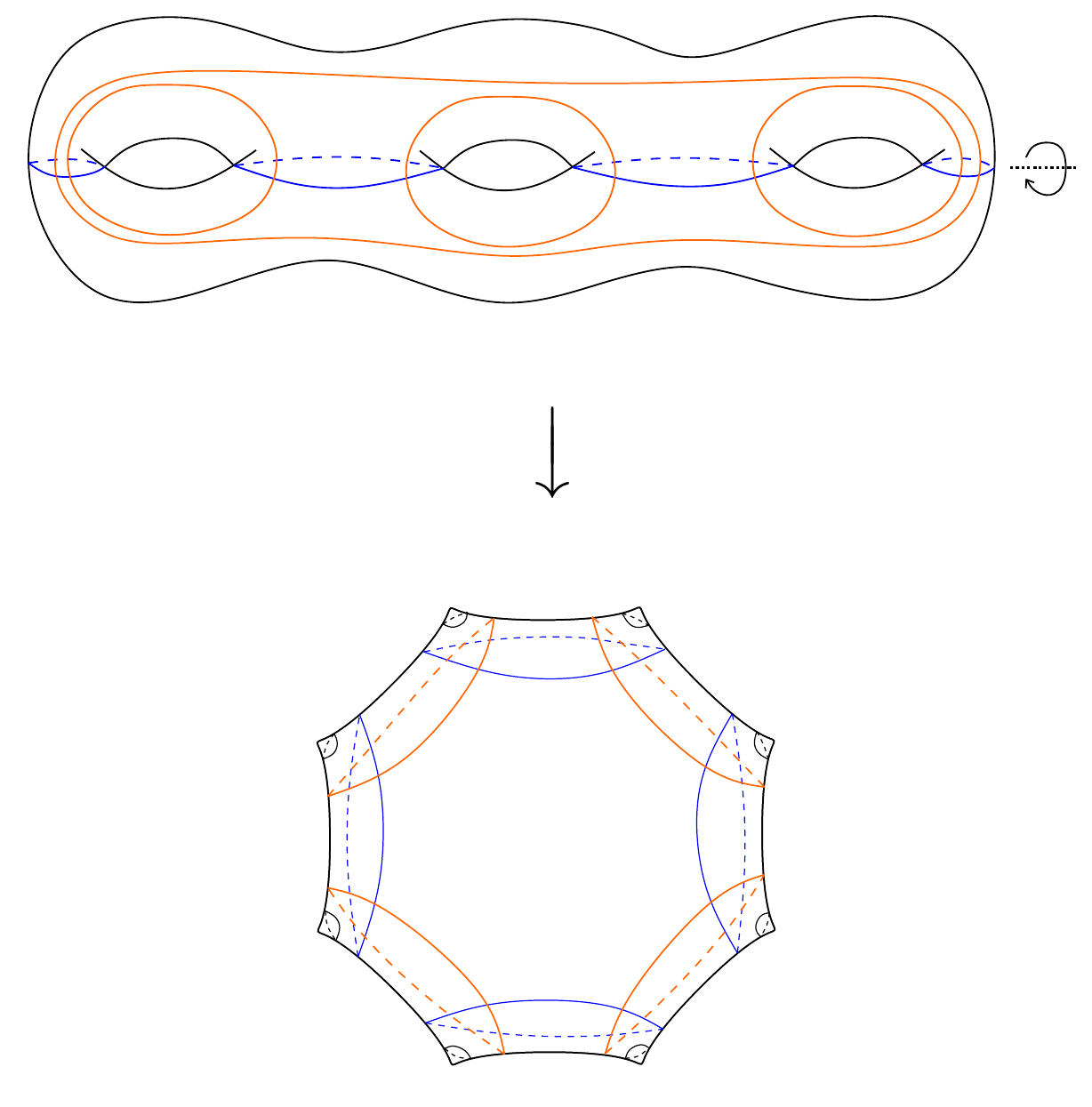}}
\end{picture}
\caption{Two-fold branched covering map from $S_{3,0}$ to $S_{0,8}$.}
\label{Fig:brancheddoublecover} 
\end{figure}

Indeed, $S_{0,2g+2}$ and $S_{g,0}$ are related by a two-fold branched covering map $S_{g,0} \to S_{0,2g+2}$ (see Figure \ref{Fig:brancheddoublecover} for an example). The $2g+2$ punctures on the sphere are the branch points. The deck transformation is the \textit{hyperelliptic involution} of $S_{g,0}$, which we denote $\iota$. Since every element of $\Map(S_{g,0})$ has a representative which commutes with $\iota$, it follows that there is a map 
\[
\Theta \from \Map(S_{g,0}) \to \Map(S_{0, 2g+2}).
\]
The kernel of the map is the cyclic group of order two generated by the involution $\iota$. Each generator for $\Map(S_{0,2g+2})$ lifts to $\Map(S_{g,0})$, so $\Theta$ is surjective. From this we have the following short exact sequence:
\[
1 \to \left< \iota \right> \to \Map(S_{g,0}) \xrightarrow{\Theta} \Map(S_{0,2g+2}),
\]
and therefore a presentation for $\Map(S_{0,2g+2})$ can be lifted to a presentation for $\Map(S_{g,0})$.

However, the map $\Theta$ is not a priori well-defined. The problem is that the elements of $\Map(S_{g,0})$ are only defined up to isotopy and these isotopies are not required to respect the hyperelliptic involution.

Let $p \from S \to X$ be a covering map of surfaces, possibly branched, possibly with boundary. We say that $f \from S \to S$ is \textit{fiber preserving} if for each $x \in X$ there is a $s \in X$ so that $f(p^{-1}(x)) = p^{-1}(y)$, ie. $f$ takes fibers to fibers.

If any two homotopic fiber-preserving mapping classes of $S$ are homotopic through fiber-preserving homeomorphisms, then we say that the covering map $p$ has the \textit{Birman-Hilden property}. Equivalently, whenever a fiber-preserving homeomorphism is homotopic to the identity, it is homotopic to the identity through fiber preserving homeomorphisms \cite{BH}.

\begin{theorem}[Birman-Hilden]\label{BirmanHilden}
Let $p \from S \to X$ be a finite-sheeted regular branched covering map where $S$ is a hyperbolic surface. Assume that $p$ is either unbranched or solvable. Then $p$ has the Birman-Hilden property.
\end{theorem}

Maclachlan and Harvey were able to give the following generalization of Theorem \ref{BirmanHilden} \cite{MH}:

\begin{theorem}[Maclachlan-Harvey]\label{MacHarv}
Let $p \from S \to X$ be a finite-sheeted regular branched covering map where $S$ is a hyperbolic surface. Then $p$ has the Birman-Hilden property.
\end{theorem}

We are able to apply Theorem \ref{MacHarv} to the branched covering map $S_{g,0} \to S_{0,2g+2}$ to find that the map $\Theta$ is well defined. Therefore, every pseudo-Anosov mapping class on $S_{0,2g+2}$ lifts to a pseudo-Anosov mapping class on $S_{g,0}$. Indeed, consider a pseudo-Anosov map $\phi$ which was found using one of the constructions in either Section \ref{Sec:Main} or Section \ref{Sec:Modifications}. This map lifts to a map $\psi$ on $S_{g,0}$. Since $\psi$ has the same local properties of $\phi$, the stable and unstable foliations for $\psi$ are the preimages under $p$ for those of $\phi$. Therefore, $\psi$ is a pseudo-Anosov mapping class on $S_{g,0}$.

Similarly, it is possible to find that $S_{0,2g+3}$ and $S_{g,2}$ are related by a two-fold branched covering map $S_{g,2} \to S_{0,2g+3}$, so we may also lift the maps from odd-times punctured spheres to surfaces with higher genus.

%---------------------------------------------------------------------------------------------------------------------%
\section{Number-Theoretic Properties}\label{Sec:Algebraic}

In this section, we will discuss the number-theoretic properties for some of the maps we are able to construct from Theorems \ref{partitions} and \ref{modification}.

The \textit{trace field} of a linear group is the field generated by the traces of its elements. In particular the trace field of a group $\Gamma \subset SL_2 (\mathbb{R})$ is the subfield of $\mathbb{R}$ generated by $tr(A)$, $A \in \Gamma$. Kenyon and Smillie proved that if the affine automorphism group of a surface contains an orientation preserving pseudo-Anosov element $f$ with largest eigenvalue $\lambda$, then the trace field is $\mathbb{Q}(\lambda + \lambda^{-1})$ \cite{KS}.

There has already been a considerable amount of research regarding the number-theoretic properties for both Penner and Thurston's constructions. For Penner's construction, Shin and Strenner were able to show that the Galois conjugates of the stretch factor are never on the unit circle \cite{SS}. For Thurston's construction, Hubert and Lanneau were able to show that the field $\mathbb{Q}(\lambda+1/\lambda)$ is totally real \cite{HL}.

In this section, we will show that some of the maps resulting from the constructions outlined in Theorem \ref{partitions} and Theorem \ref{modification} are such that the Galois conjugates of the stretch factor are on the unit circle, and that the field $\mathbb{Q}(\lambda+1/\lambda)$ is not totally real. This result shows that the construction outlined above cannot come from neither Penner nor Thurston’s construction. However, we show it is also possible to find maps resulting from these constructions where $\mathbb{Q} (\lambda + 1/\lambda)$ is totally real and that the Galois conjugates of the stretch factors are not on the unit circle, where $\mathbb{Q} (\lambda + 1/\lambda)$ is totally real and that the Galois conjugates of the stretch factors are on the unit circle, or that $\mathbb{Q} (\lambda + 1/\lambda)$ is not totally real and the Galois conjugates of the stretch factors are not on the unit circle.

We begin by examining the algebraic properties of the first map introduced in Example \ref{example6punctures}. For this map, we will see that the Galois conjugates are never on the unit circle, and that the field $\mathbb{Q}(\lambda+1/\lambda)$ is totally real.

\begin{example}
Consider the pseudo-Anosov map $\phi$ on $S_{0,6}$ introduced in Example \ref{example6punctures}. The characteristic polynomial of this map is
\[
p_{\phi} = (x-1)^2 (x+1)^2 ( x^2 - 18x +1).
\]
We notice that the leading eigenvalue $\lambda_{\phi}$ is a root of the factor $p_{\phi,\lambda}(x) = x^2 - 18x +1$, which is an irreducible polynomial with real roots. Since $1<\lambda_{\phi} \in \mathbb{R}$ is not on the unit circle, and therefore $\lambda^{-1}$ is also not on the unit circle, that the Galois conjugates of the stretch factor are not on the unit circle.

To show that $\mathbb{Q}(\lambda+1/\lambda)$ is totally real, we notice that we are able to write
\[
\frac{p_{\phi}}{x} = \left( x+ \frac{1}{x} \right) -18 = q\left( x+\frac{1}{x} \right).
\]
By considering the roots of $q(y) = y-18$, we notice that the only root is $y=18$ which implies that the field $\mathbb{Q}(\lambda+1/\lambda)$ is totally real. 
\end{example}

Next, we will provide an example where the Galois conjugates are never on the unit circle, and that the field $\mathbb{Q}(\lambda+1/\lambda)$ is not totally real.

\begin{example}
We consider the map $\bar{\phi '}$ from Example \ref{example7punctures}. We compute that the characteristic polynomial of this map is 
\[
p_{\bar{\phi}} = (x + 1) (x^3 - 15x^2 + 7x - 1) (x^3 - 7x^2 + 15x - 1).
\]
We notice that $\lambda$ is a root of the polynomial $p_{\bar{\phi},\lambda}(x) = x^3 - 15x^2 + 7x - 1$. The roots of this polynomial are
\begin{equation*}
5+\frac{1}{3}\sqrt[3]{2916-12\sqrt{19}}+\left( \frac{2}{3} \right)^{2/3}\sqrt[3]{243+\sqrt{93}},
\end{equation*}
\begin{equation*}
5+\frac{1}{3}(-1+i \sqrt{3})\sqrt[3]{2916-12\sqrt{19}} - \frac{(1+i\sqrt{3})(\sqrt[3]{1/2(243+\sqrt{93}}))}{3^{2/3}},
\end{equation*}
and
\begin{equation*}
5-\frac{1}{3}(-1+i \sqrt{3})\sqrt[3]{2916-12\sqrt{19}} + \frac{(1+i\sqrt{3})(\sqrt[3]{1/2(243+\sqrt{93}}))}{3^{2/3}}.\\
\end{equation*}
By unique factorization, we see that $p_{\bar{\phi},\lambda}(x) = x^3 - 15x^2 + 7x - 1$ is irreducible over $\mathbb{Q}$. None of the roots of $p_{\bar{\phi},\lambda}(x)$ are on the unit circle, so we have that there are no Galois conjugates of the stretch factor on the unit circle. We now consider the polynomial
\[
\left( \frac{1}{x^3}\right) (x^3 - 15x^2 + 7x - 1) (x^3 - 7x^2 + 15x - 1) = x^3-22x^2+127x-276.
\]
We are able to rewrite this polynomial as
\[
q\left( x+\frac{1}{x}\right) = \left( x+\frac{1}{x} \right)^3 -22 \left( x+\frac{1}{x} \right)^2 +124 \left( x+\frac{1}{x} \right) -232 .
\]
We calculate that the roots of the polynomial $q(y)$ are
\begin{equation*}
\frac{1}{3} \left( 22 + \sqrt[3]{1801-9\sqrt{26554}} + \sqrt[3]{1801+9\sqrt{26554}} \right),
\end{equation*}
\begin{equation*}
\frac{1}{6} \left( 44 + i(\sqrt{3}+i) \sqrt[3]{1801-9\sqrt{26554}} + (-1-i\sqrt{3}) \sqrt[3]{1801+9\sqrt{26554}} \right),
\end{equation*}
and
\begin{equation*}
\frac{1}{6} \left( 44 + (-1-i\sqrt{3}) \sqrt[3]{1801-9\sqrt{26554}} + i(\sqrt{3}+i) \sqrt[3]{1801+9\sqrt{26554}} \right).
\end{equation*}
By unique factorization, $q(y)$ is irreducible. Additionally, since two of the roots are imaginary, the field $\mathbb{Q}(\lambda+1/\lambda)$ is not totally real.
\end{example}

We now provide an example where there are Galois conjugates of the stretch factor on the unit circle, and that the field $\mathbb{Q}(\lambda+1/\lambda)$ is totally real.

\begin{example}
We now consider the pseudo-Anosov map $\psi = D_{5}^2D_{4}^2D_{8}^2D_{3}^2D_{7}^2D_{2}^2D_{6}^2D_{1}^2$, which is the first map from Example \ref{example6punctures} with the modification from Theorem \ref{modification} applied twice so that there are two partitions with one element each. The train track $\eta$ where $\psi (\eta)$ is carried by $\psi$ is depicted in Figure \ref{Fig:pA632mod2} .
\begin{figure}[h]
\setlength{\unitlength}{0.01\linewidth}
\begin{picture}(100,27)
\put(37,-6){\includegraphics[width=28\unitlength]{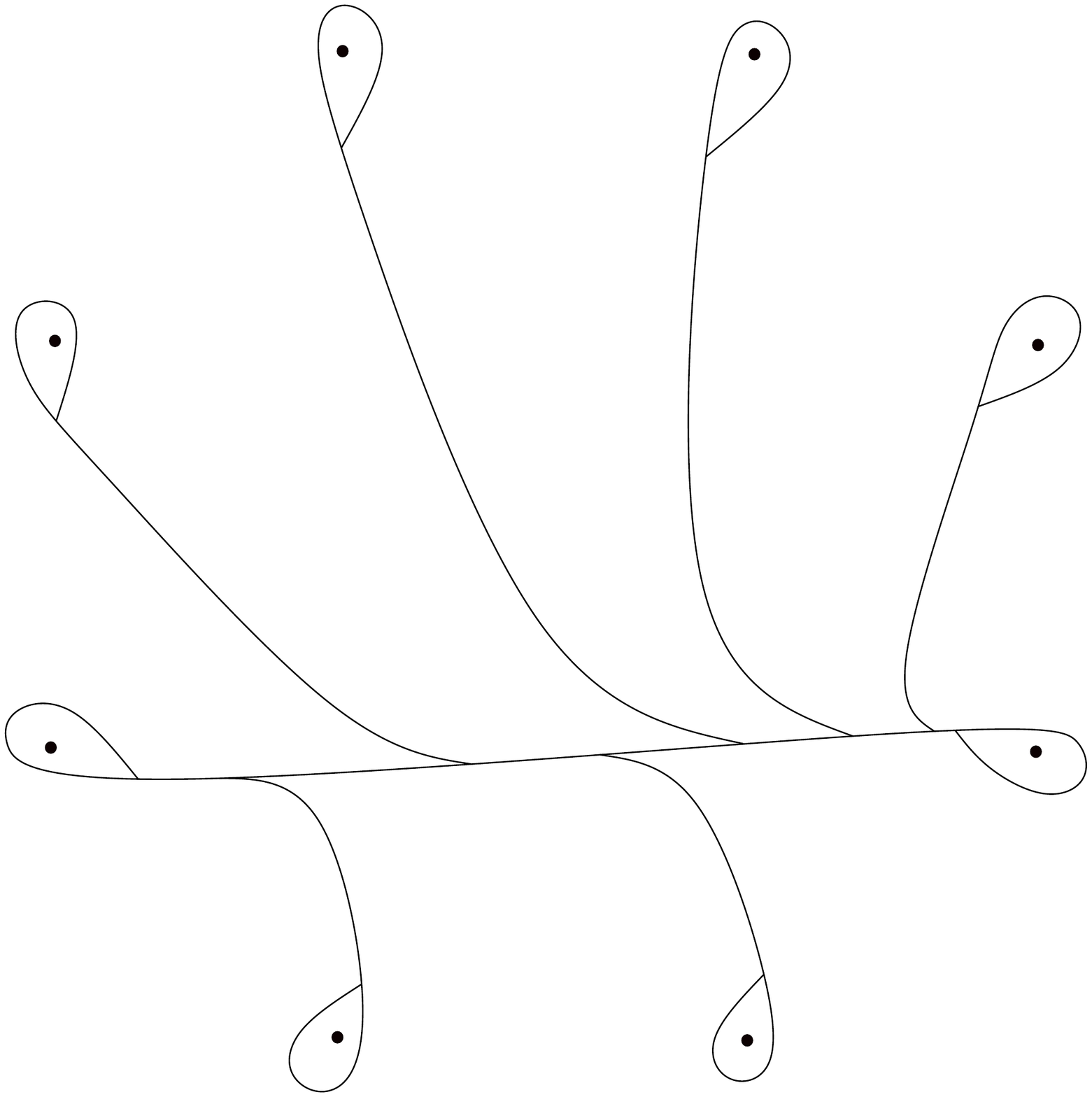}}
\end{picture} 
\caption{The train track associated to $\psi$.}
\label{Fig:pA632mod2} 
\end{figure}
The matrix associated to this map is
\[ M= \left( \begin{array}{cccccccc}
3 & 2 & 0 & 0 & 0 & 0 & 0 & 2 \\
6 & 3 & 2 & 4 & 0 & 0 & 0 & 4 \\
12 & 6 & 3 & 6 & 0 & 0 & 0 & 8\\
0 & 0 & 2 & 3 & 2 & 0 & 0 & 0 \\
0 & 0 & 4 & 6 & 3 & 2 & 0 & 0 \\
0 & 0 & 8 & 12 & 6 & 3 & 2 & 0 \\
4 & 0 & 16 & 24 & 12 & 6 & 3 & 2 \\
6 & 0 & 32 & 48 & 24 & 12 & 6 & 3 \end{array} \right)\] 
which has the characteristic polynomial
\[
p(x) = (x + 1)^4 (x^4 - 28x^3 + 6x^2 - 28x + 1).
\]
Our leading eigenvalue $\lambda$ is a root of $p_{\lambda,\psi}(x) = x^4 - 28x^3 + 6x^2 - 28x + 1$. The roots of this polynomial are
\[
\lambda^{-1} = 7+4\sqrt{3}-2\sqrt{24+14\sqrt{3}},
\]
\[
\lambda = 7+4\sqrt{3}+2\sqrt{24+14\sqrt{3}},
\]
\[
x_1 = 7-4\sqrt{3}-2i\sqrt{14\sqrt{3}-24},
\]
and
\[
x_2 = 7-4\sqrt{3}+2i\sqrt{14\sqrt{3}-24}.
\]
By unique factorization, we know that $p_{\lambda,\psi}(x) = x^4 - 28x^3 + 6x^2 - 28x + 1$ is irreducible over $\mathbb{Q}$. Notice that $|x_1| = 1$, and $|x_2|$=1, which we can verify by direct computation or by applying by Theorem 1 of \cite{KM}. This implies that there are Galois conjugates of the stretch factor on the unit circle. We now show that the field $\mathbb{Q}(\lambda+1/\lambda)$ is totally real by writing
\[
\frac{p_{\lambda,\psi}}{x} = \left( x+ \frac{1}{x} \right)^2 -28\left( x+ \frac{1}{x} \right) + 4 = q\left( x+\frac{1}{x} \right).
\]
We notice that the roots of $q(y) = y^2 - 28y +4$ are
\[
14-8\sqrt{3}
\]
and
\[
14+8\sqrt{3},
\]
which implies that $q(y)$ is irreducible by unique factorization. Additionally, since both roots are real we find that the field $\mathbb{Q}(\lambda+1/\lambda)$ is totally real.
\end{example}

Lastly, we provide an example where there are Galois conjugates of the stretch factor on the unit circle, and where the field $\mathbb{Q}(\lambda+1/\lambda)$ is not totally real.

\begin{example}\label{uniquepA}
We consider the second map from Example \ref{example6punctures} and apply the modification from Theorem \ref{modification} twice so that there are two partitions with one element each. This induces the map $\psi' = D_{4}^2D_{3}^2D_{8}^2D_{6}^2D_{2}^2D_{7}^2D_{5}^2D_{1}^2$. The train track $\eta'$ where $\psi' (\eta')$ is carried by $\psi'$ is depicted in Figure \ref{Fig:pA623mod2} .
\begin{figure}[h]
\setlength{\unitlength}{0.01\linewidth}
\begin{picture}(100,28)
\put(37,-6){\includegraphics[width=28\unitlength]{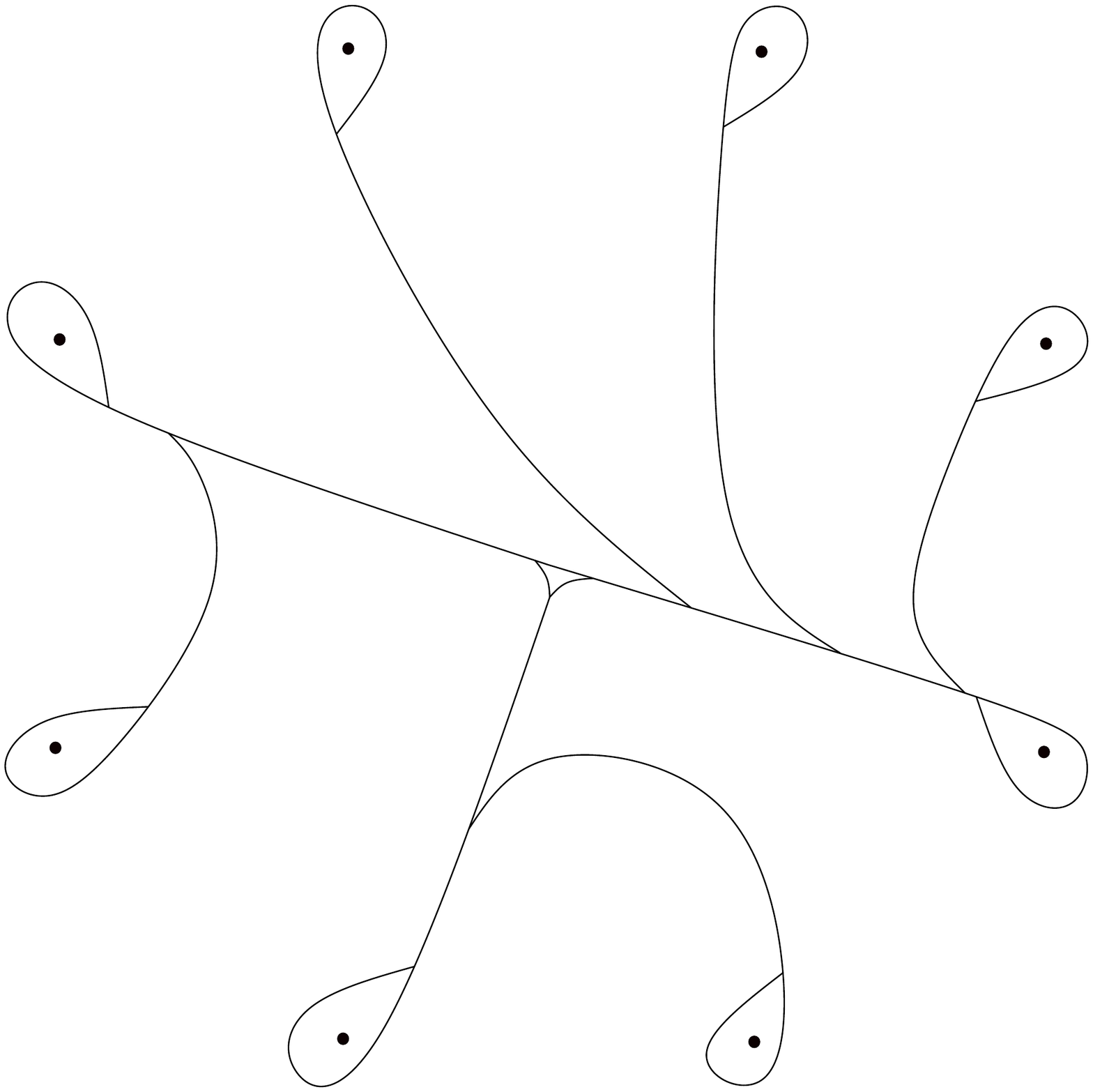}}
\end{picture} 
\caption{The train track associated to $\psi$.}
\label{Fig:pA623mod2} 
\end{figure}
The matrix associated to this map is
\[ M= \left( \begin{array}{cccccccc}
3 & 2 & 0 & 0 & 0 & 0 & 0 & 2 \\
6 & 3 & 2 & 0 & 0 & 0 & 0 & 4 \\
12 & 6 & 3 & 2 & 4 & 0 & 0 & 8\\
24 & 12 & 6 & 3 & 6 & 0 & 0 & 16 \\
0 & 0 & 0 & 2 & 3 & 2 & 4 & 0 \\
0 & 0 & 0 & 4 & 6 & 3 & 6 & 0 \\
4 & 0 & 0 & 0 & 0 & 2 & 3 & 2 \\
6 & 0 & 0 & 0 & 0 & 4 & 6 & 3 \end{array} \right)\] 
which has the characteristic polynomial 
\[
p_{\psi '}(x) = x^8 - 24x^7 + 156x^6 - 424x^5 - 186x^4 - 424x^3 + 156x^2 - 24x + 1.
\]
This polynomial is irreducible, proof located in Appendix \ref{irreducibility}, and by Theorem 1 of \cite{KM} we find that there are roots of this polynomial which are on the unit circle. We now write
\[
\frac{p_{\psi '}}{x^4} = \left( x+\frac{1}{x} \right) ^4 -24 \left( x+\frac{1}{x} \right) ^3 +152 \left( x+\frac{1}{x} \right) ^2-352 \left( x+\frac{1}{x} \right) -496.
\]
We are able to prove that $q(y) = y^4 -24y^3 + 152y^2 -352y - 496$ is an irreducible polynomial as follows. By Gauss's lemma, a primitive polynomial is irreducible over the integers if and only if it is irreducible over the rational numbers. Since $q(y)$ is primitive, it suffices to show that $q(y)$ is irreducible over the integers. The rational root theorem gives us that $q(y)$ has no roots, so if it is reducible then $q(y) = (y^2+ay+b)(y^2+cy+d)$. Therefore, suppose that $q(y) = (y^2+ay+b)(y^2+cy+d)$. Expanding gives rise to the system of equations
\begin{equation}
\begin{aligned}
a + c &= -24\\
ac+ b + d &= 152\\
ad + bc &= -352\\
bd &= -496
\end{aligned}
\end{equation}
Substituting $a = -24 - c$ and $b = \frac{-496}{d}$ into the second and third equations give
\begin{equation}
\begin{aligned}
-24c - c^2 - \frac{496}{d} + d &= 152\\
(-24 - c)d - \frac{496c}{d} & = -352
\end{aligned}
\end{equation}
We solve for $c$ in the second equation to find
\[
c = \frac{24d^2 - 352 d}{-d^2 - 496}.
\] 
Substituting this into the first equation gives
\[
-24 \left( \frac{24d^2 - 352 d}{-d^2 - 496} \right) d - \left(\frac{24d^2 - 352 d}{-d^2 - 496}\right)^2d + d^2 - 152 d - 496 = 0,
\]
which has no integer roots. Therefore, there is no $d$ satisfying the conditions we require, therefore $q(y)$ is irreducible.
Finally, by using the formulas for the roots of a quartic equation, we see that $q(y)$ has two imaginary roots, therefore the field $\mathbb{Q}(\lambda+1/\lambda)$ is not totally real.
\end{example}

\begin{proof}[Proof of Theorem \ref{uniqueconstruction}]
Consider the map $\psi'$ from Example \ref{uniquepA}. For Thurston's construction Hubert and Lanneau proved that the field $\mathbb{Q}(\lambda+1/\lambda)$ is always totally real \cite{HL}, and for Penner's construction Shin and Strenner proved the Galois conjugates of the stretch factor are never on the unit circle \cite{SS}. In Example \ref{uniquepA}, it was shown that the map $\psi'$ has a trace field which is not totally real and there are Galois conjugates of the stretch factor on the unit circle. Therefore, this mapping class is unable to come from either Thurston's or Penner's constructions.
\end{proof}

%---------------------------------------------------------------------------------------------------------------------%

\begin{appendices}

\section{Irreducibility of Polynomials}\label{irreducibility}

Here we present the proof of the irreducibility of the polynomial $p_{\psi '}(x) = x^8 - 24x^7 + 156x^6 - 424x^5 - 186x^4 - 424x^3 + 156x^2 - 24x + 1$ over $\mathbb{Q}$.

\begin{theorem}
The polynomial
\[
p(x) = x^8 - 24x^7 + 156x^6 - 424x^5 - 186x^4 - 424x^3 + 156x^2 - 24x + 1
\]
is irreducible over $\mathbb{Q}$.
\end{theorem}

\begin{proof}
By Gauss's lemma, a primitive polynomial is irreducible over the integers if and only if it is irreducible over the rational numbers. Since $p(x)$ is primitive, it suffices to show that $p(x)$ is irreducible over the integers.

By the rational root theorem, the only rational roots of the polynomial may be $\pm 1$, however, neither are a root of $p(x)$. This implies that the polynomial $p(x)$ has no linear factors.

Additionally, note that for any palindromic polynomial, if $x$ is a root then its inverse must also be a root of the polynomial. This is a fact which will be used often.

Suppose now that $p(x)$ factors as $p(x) = f(x) g(x)$ where $deg(f(x)) = 2$ and $deg(g(x))=6$. Since we are showing irreducibility over the integers and the leading coefficient and the constant term of $p(x)$ are $1$, we see that $f(x) = x^2 + ax + 1$ or $f(x) = x^2 + ax -1$.

Suppose first that $f(x) = x^2 + ax + 1$ and $g(x) = x^6 + bx^5 + cx^4 + dx^3 + ex^2 + fx + 1$. This gives rise to the system of equations
\begin{equation}
\begin{aligned}
a+b &= -24\\
c+ab+1 &= 156\\
d+ac+b &= -424\\
e+ad+c &= -186\\
f+ae+d &= -424\\
a+af+e &= 156\\
a+f &= -24
\end{aligned}
\end{equation}
Using the first and last equations in the system, we can find that $b=f$, and therefore $a = -24 - f$. Inputting this information into the second equation will give that $c = 155 + 24f +f^2$, and similarly $e = 155 + 24f + f^2$. We then input this information into the third equation where we find that $d = f^3+48f^2+720f + 3296$. Inputting all of these equations into the fourth equation, we see that
\[
f^4 + 72f^3 + 1880f^2 + 20792f + 78584 = 0.
\]
By the rational root theorem, there are no rational roots of the equation above, which implies that $f$ is not a rational number, which is a contradiction.

Recall that for a palindromic polynomial, the inverse of each root is also a root of the polynomial. If $f(x) = x^2 + ax - 1$, then the roots of $f$ are not able to be inverses of each other as the constant term is $-1$. Therefore, the inverses of the two roots of $f(x)$ are roots of $g(x)$, which implies that $g(x)$ is divisible by $\bar{f}(x) = x^2 -ax -1$, so $p(x) = f(x) \bar{f}(x) h(x)$, where $h(x) = x^4 + bx^3 + cx^2 + dx + 1$. Expanding this, we find the system of equations
\begin{equation}
\begin{aligned}
b &= -24\\
-a^2 + c - 2 &= 156\\
-a^2b -2b +d &= -424\\
-a^2c - 2c + 2 &= -186\\
-a^2d +b -2d &= -424\\
-a^2 +c - 2 &= 156\\
d &= -24
\end{aligned}
\end{equation}
Inputting that $b=-24$ and $d=-24$ into the third equation immediately implies that $a^2 = -\frac{56}{3}$, which implies $a$ is not a rational number which is a contradiction.

Now suppose that $p(x)$ factors as $p(x) = f(x)g(x)$ where $deg(f(x)) = 3$ and $deg(g(x))=5$. Suppose that two of the roots of $f$ were each others’ inverses. Then the third root of $f$ would have to be $\pm 1$ since the constant term of $f(x)$ must be either $\pm1$. However, this is not possible as neither of $\pm1$ are roots of $p(x)$. Therefore, if $x$ is a root of $f(x)$, the inverse of $x$ will not be a root of $f(x)$. This implies that all the inverses of roots of $f(x)$ must be a root of $g(x)$. This implies that we have two possible cases. Either
\[
\text{i) } p(x) = (x^3+ax^2+bx+1)(x^3+bx^2+ax+1)(x^2+cx+1)
\]
or
\[
\text{ii) } p(x) = (x^3+ax^2+bx-1)(x^3-bx^2-ax-1)(x^2+cx+1).
\]

Suppose first that 
\[
p(x) = (x^3+ax^2+bx+1)(x^3+bx^2+ax+1)(x^2+cx+1).
\]
This gives rise to the following system of equations:
\begin{equation}
\begin{aligned}
a + b + c &= -24\\
ab + bc + ac + a + b + 1 &= 156\\
a^2 + b^2 + abc + ac + bc + a + b + 2 &= -424\\
a^2c + b^2c + 2ab + 2a + 2b + 2c &= -186\\.
\end{aligned}
\end{equation}
The first equation implies that $a = -24 - b - c$, which input into the second equation implies 
\[
-b^2-bc-24b-c^2-25c-179=0
\]
or that
\[
b = \frac{1}{2}(-\sqrt{-3c^2-52c-140}-c-24)
\]
or 
\[
b = \frac{1}{2}(\sqrt{-3c^2-52c-140}-c-24).
\]
We notice that this implies that $-14<c<-\frac{10}{3}$. Now we input $a = -24 - b - c$ into the third equation to find
\begin{equation}\label{irred1}
b^2(2-c)+(48-22c-c^2)b+23c+978=0.
\end{equation}
Inputting 
\[
b = \frac{1}{2}(-\sqrt{-3c^2-52c-140}-c-24)
\]
into Equation \ref{irred1} we find that the only real root of the equation is
\[
c = \frac{1}{3} \left( -23 - \frac{73}{\sqrt[3]{4805-96\sqrt{2463}}} - \sqrt[3]{4805-96\sqrt{2463}} \right),
\]
which is less than $-14$ and out of the domain of $c$. Similarly, by inputting
\[
b = \frac{1}{2}(\sqrt{-3c^2-52c-140}-c-24)
\]
into Equation \ref{irred1} we find that the only real root of the equation is
\[
c = \frac{1}{3} \left( -23 - \frac{73}{\sqrt[3]{4805-96\sqrt{2463}}} - \sqrt[3]{4805-96\sqrt{2463}} \right),
\]
which is less than $-14$ and out of the domain of $c$. 

Now suppose that
\[
p(x) = (x^3+ax^2+bx-1)(x^3-bx^2-ax-1)(x^2+cx+1).
\]
Expanding, we find the system of equations
\begin{equation}
\begin{aligned}
a - b + c &= -24\\
-ab - a + ac - bc + b + 1 &= 156\\
-a^2 - abc - ac + a - b^2 + bc - b &= -424\\
-a^2c - 2ab + b^2c - 2a &= -186\\
-a^2 - abc - ac - a - b^2 - bc - b &= -424\\
-ab - a - ac - bc - b + 1 &= 156\\
-a + b + c &= -24
\end{aligned}
\end{equation}
Adding together the first and last equations, we find that $c = -24$. Inputting this back into the first equation, we find that $a = b$. Inputting this information into the second equation, we find that $a^2 = -155$, which is not rational and therefore a contradiction.

Finally, suppose that $p(x)$ factors as $p(x) = f(x)g(x)$ where $deg(f(x)) = 4$ and $deg(g(x))=4$. Either, we have
\[
\text{i) } p(x) = (x^4+ax^3+bx^2+cx+1)(x^4+dx^3+ex^2+fx+1)
\]
or
\[
\text{ii) } p(x) = (x^4+ax^3+bx^2+cx-1)(x^4+dx^3+ex^2+fx-1).
\]

First we consider the case where $p(x) = (x^4+ax^3+bx^2+cx+1)(x^4+dx^3+ex^2+fx+1)$. Since $p(x)$ is a palindrome, this implies that $x^8 f(\frac{1}{x}) = p(x)$, therefore we need to have $(x^4+cx^3+bx^2+ax+1)(x^4+fx^3+ex^2+dx+1)=(x^4+ax^3+bx^2+cx+1)(x^4+dx^3+ex^2+fx+1)$. This gives rise to the following two potential subcases. Either
\[
p(x) = (x^4+ax^3+bx^2+ax+1)(x^4+dx^3+ex^2+dx+1)
\]
or
\[
p(x) = (x^4+ax^3+bx^2+cx+1)(x^4+cx^3+bx^2+ax+1).
\]
Suppose first that $p(x) = (x^4+ax^3+bx^2+ax+1)(x^4+dx^3+ex^2+dx+1)$. Expanding, we find the system of equations
\begin{equation}
\begin{aligned}
a + d &= -2\\
ad + b + e &= 156\\
a + ae + bd + d &= -424\\
2ad + be + 2 &= -186\\
\end{aligned}
\end{equation}
The first equation implies that $a = -24 - d$. Inputting this into the second equation we will find that $b = 156 - e + (24+d)d$. Now we can use these two equations along with the third equation to find that
\[
e = d^2 + 24d + \frac{418}{3}.
\]
Inputting all of this into the final equation we find that either
\[
d = -12 - i\sqrt{\frac{1783}{66}}
\]
or
\[
d = -12 + i\sqrt{\frac{1783}{66}},
\]
both of which are imaginary, which is a contradiction.

Now suppose that $p(x) = (x^4+ax^3+bx^2+cx+1)(x^4+cx^3+bx^2+ax+1)$. Expanding, we find the system of equations
\begin{equation}
\begin{aligned}
a + c &= -24\\
ac + 2b &= 156\\
a + ab + bc + c &= -424\\
a^2 + b^2 + c^2 + 2 &= -186\\
\end{aligned}
\end{equation}
The first equation gives $a = -24 - c$, which input into the second equation gives $b = \frac{156+(24+c)c}{2}$. Inputting these equations into the third and solving for c gives
\[
c = \frac{8}{\sqrt{3}} - 12
\]
which is an irrational number, which is a contradiction.

Finally, consider the case where $p(x) = (x^4+ax^3+bx^2+cx-1)(x^4+dx^3+ex^2+fx-1)$. Expanding, we find the system of equations
\begin{equation}
\begin{aligned}
a + d &= -24\\
ad + b + e &= 156\\
ae + bd + c + f &= -424\\
af + be + cd &= -186\\
-a + bf + ce - d &= -424\\
-b + cf - e &= 156\\
-c - f &= -24
\end{aligned}
\end{equation}
The first equation gives $a = -24 - d$, which input into the second equation gives $156 - e + (24+d)d$. The last equation gives $c = 24 - f$, which all input into the sixth equation gives
\[
f = 12 \pm \sqrt{-d^2 -24d -168}.
\]
This implies that we must have a value of $d$ such that $-d^2 -24d -168 > 0$, which implies that $-\frac{1}{24}(d+12)^2>1$. However, there is no real value of $d$ which satisfies this inequality.
\end{proof}

\end{appendices}

%---------------------------------------------------------------------------------------------------------------------%
%---------------------------------------------------------------------------------------------------------------------%
%---------------------------------------------------------------------------------------------------------------------%

%---------------------------------------------------------------------------------------------------------------------%
%---------------------------------------------------------------------------------------------------------------------%
%---------------------------------------------------------------------------------------------------------------------%

\end{document}